\documentclass[11pt]{article}

\usepackage{amssymb,amsmath,amsfonts,amsthm}

\usepackage{caption2}

\renewcommand{\subsection}{\subsubsection}

\newtheorem{theorem}{Theorem}[section]

\newtheorem{lemma}{Lemma}[section]
\newtheorem{proposition}{Proposition}[section]

\newtheorem{remark}{Remark}[section]

\begin{document}
\title{\bf Well-posedness of the linearized problem \\ for MHD contact discontinuities}
\author{{\bf Alessandro Morando}\\
DICATAM, Sezione di Matematica, Universit\`a di Brescia \\ Via Valotti, 9, 25133 Brescia, Italy\\
E-mail: alessandro.morando@ing.unibs.it
\and
{\bf Yuri Trakhinin}\footnote{Supported in part by the
Landau Network--Centro Volta--Cariplo Foundation and the Ministry of Education and Science of the Russian Federation (contract 14.B37.21.0355).}\\
Sobolev Institute of Mathematics, Koptyug av. 4, 630090 Novosibirsk, Russia\\
E-mail: trakhin@math.nsc.ru
\and
{\bf Paola Trebeschi}\\
DICATAM, Sezione di Matematica, Universit\`a di Brescia \\ Via Valotti, 9, 25133 Brescia, Italy\\
E-mail: paola.trebeschi@ing.unibs.it
}

\date{
}
%
%
\maketitle
\begin{abstract}
We study the free boundary problem for contact discontinuities in ideal compressible magnetohydrodynamics (MHD).  They are characteristic discontinuities with no flow across the discontinuity for which the pressure, the magnetic field and the velocity are continuous whereas the density and the entropy may have a jump. Under the Rayleigh-Taylor sign condition $[\partial p/\partial N]<0$ on the jump of the normal derivative of the pressure satisfied at each point of the unperturbed contact discontinuity, we prove the well-posedness in Sobolev spaces of the linearized problem for 2D planar MHD flows.
\end{abstract}
\section{Introduction}
\label{intro}
We consider the equations of ideal compressible MHD:
\begin{equation}\label{1}
\left\{
\begin{array}{l}
 \partial_t\rho  +{\rm div}\, (\rho {v} )=0,\\[6pt]
 \partial_t(\rho {v} ) +{\rm div}\,(\rho{v}\otimes{v} -{H}\otimes{H} ) +
{\nabla}q=0, \\[6pt]
 \partial_t{H} -{\nabla}\times ({v} {\times}{H})=0,\\[6pt]
 \partial_t\bigl( \rho e +{\textstyle \frac{1}{2}}|{H}|^2\bigr)+
{\rm div}\, \bigl((\rho e +p){v} +{H}{\times}({v}{\times}{H})\bigr)=0,
\end{array}
\right.
\end{equation}
where $\rho$ denotes density, $v\in\mathbb{R}^3$ plasma velocity, $H \in\mathbb{R}^3$ magnetic field, $p=p(\rho,S )$ pressure, $q =p+\frac{1}{2}|{H} |^2$ total pressure, $S$ entropy, $e=E+\frac{1}{2}|{v}|^2$ total energy, and  $E=E(\rho,S )$ internal energy. With a state equation of gas, $\rho=\rho(p ,S)$, and the first principle of thermodynamics, \eqref{1} is a closed system for the unknown $ U =U (t, x )=(p, v,H, S)$.

System (\ref{1}) is supplemented by the divergence constraint
\begin{equation}
{\rm div}\, {H} =0
\label{2}
\end{equation}
on the initial data ${U} (0,{x} )={U}_0({x})$. As is known, taking into account \eqref{2}, we can easily symmetrize system \eqref{1} by rewriting it in the nonconservative form
\begin{equation}
\left\{
\begin{array}{l}
{\displaystyle\frac{1}{\rho c^2}\,\frac{{\rm d} p}{{\rm d}t} +{\rm div}\,{v} =0,\qquad
\rho\, \frac{{\rm d}v}{{\rm d}t}-({H},\nabla ){H}+{\nabla}q  =0 ,}\\[9pt]
{\displaystyle\frac{{\rm d}{H}}{{\rm d}t} - ({H} ,\nabla ){v} +
{H}\,{\rm div}\,{v}=0},\qquad
{\displaystyle\frac{{\rm d} S}{{\rm d} t} =0},
\end{array}\right. \label{3}
\end{equation}
where $c^2 =1/\rho_p$ is the square of the sound speed, $\rho_p=\partial\rho /\partial p$, ${\rm d} /{\rm d} t =\partial_t+({v} ,{\nabla} )$ and by $(\ ,\ )$ we denote the scalar product. Equations (\ref{3}) form the symmetric system
\begin{equation}
A_0(U )\partial_tU+\sum_{j=1}^3A_j(U )\partial_jU=0
\label{3Dsys}
\end{equation}
which is hyperbolic if the matrix  $A_0= {\rm diag} (\rho_p/\rho ,\rho ,\rho ,\rho , 1,1,1,1)$ is positive definite, i.e.,
\begin{equation}
\rho  >0,\quad \rho_p >0. \label{5}
\end{equation}
The symmetric matrices $A_j$ can be easily written down.

Within this paper we assume that the plasma obeys the state equation of a polytropic gas
\begin{equation}
\rho(p,S)= A p^{\frac{1}{\gamma}} e^{-\frac{S}{\gamma}}, \qquad A>0,\quad \gamma>1.
\label{polgas}
\end{equation}
In this case $c^2 =\gamma p/\rho$ and \eqref{5} becomes equivalent to\footnote{In fact, for us it is only important that for a polytropic gas the value $\rho c^2=\gamma p$ is continuous if the pressure is continuous, i.e., our results hold true for state equations for which the value $\rho c^2$ has the same property.}
\begin{equation}\label{Hiper2}
\rho>0, \quad p>0.
\end{equation}
Moreover, we manage to carry out the well-posedness analysis for contact discontinuities (see their definition just below) only for 2D planar MHD flows, i.e., when the space variables $x=(x_1,x_2)\in \mathbb{R}^2$ and the velocity and the magnetic field have only two components: $v=(v_1,v_2)\in \mathbb{R}^2$,
$H=(H_1,H_2)\in \mathbb{R}^2$. In the 2D planar case and for a polytropic gas, the MHD system \eqref{3} reads
\begin{equation}
\label{4}
A_0(U )\partial_tU+A_1(U )\partial_1U+A_2(U )\partial_2U=0
\end{equation}
with $A_0= {\rm diag} (1/(\gamma p) ,\rho ,\rho ,1,1,1)$ and
\[
A_1=\left( \begin{array}{cccccc} \frac{v_1}{\gamma p}& 1 & 0 & 0 & 0 & 0\\[6pt]
1 & \rho v_1 & 0 & 0& H_2& 0 \\
0& 0& \rho v_1 & 0& -H_1& 0 \\
0& 0& 0& v_1 & 0& 0\\
0& H_2& - H_1& 0& v_1 & 0\\
0& 0& 0& 0& 0& v_1
\end{array} \right),\quad
A_2=\left( \begin{array}{cccccc} \frac{v_2}{\gamma p}& 0 & 1& 0 & 0 & 0\\[6pt]
0 & \rho v_2 & 0 & - H_2& 0& 0 \\
1& 0& \rho v_2 & H_1& 0& 0 \\
0& -H_2& H_1& v_2 & 0& 0\\
0& 0& 0& 0& v_2 & 0\\
0& 0& 0& 0& 0& v_2
\end{array} \right).
\]

We consider the MHD equations \eqref{1} for $t\in [0,T]$ in the unbounded space domain $\mathbb{R}^3$ and suppose that $\Gamma (t)=\{ x_1-\varphi (t,x')=0\}$ is a smooth hypersurface in $[0,T]\times\mathbb{R}^3$, where
$x'=(x_2,x_3)$ are tangential coordinates. We assume that $\Gamma (t)$ is a surface of strong discontinuity for the conservation laws (\ref{1}), i.e., we are interested in solutions of (\ref{1}) that are smooth on either side of $\Gamma (t)$. To be weak solutions of (\ref{1}) such piecewise smooth solutions
should satisfy the MHD Rankine-Hugoniot conditions (see, e.g., \cite{LL})
\begin{equation}
\left\{
\begin{array}{l}
[\mathfrak{j}]=0,\quad [H_{\rm N}]=0,\quad \mathfrak{j}\left[v_{\rm N}\right] +
[q]=0,\quad \mathfrak{j}\left[{v}_{\tau}\right]=H_{\rm N}[{H}_{\tau}],\\[6pt]
H_{\rm N}[{v}_{\tau}]=\mathfrak{j}\left[{H}_{\tau}/\rho\right],\quad \mathfrak{j}\left[
e+\frac{1}{2}(|{H} |^2/\rho )\right] + \left[qv_{\rm N} -H_{\rm N}({H} ,{v} )\right] =0
\end{array}
\right.
\label{6}
\end{equation}
at each point of $\Gamma$, where $[g]=g^+|_{\Gamma}-g^-|_{\Gamma}$ denotes the jump of $g$, with $g^{\pm}:=g$ in the domains
\[
\Omega^{\pm}(t)=\{\pm (x_1- \varphi (t,x'))>0\},
\]
and
\[
\mathfrak{j}^{\pm}=\rho (v_{\rm N}^{\pm}-\partial_t\varphi),\quad v_{\rm N}^{\pm}=({v}^{\pm} ,{N}),\quad H_{\rm N}=({H}^{\pm} ,{N}),\quad {N}=(1,-\partial_2\varphi,-\partial_3\varphi ),
\]
\[
{v}^{\pm}_{\tau}=(v^{\pm}_{\tau _1},v^{\pm}_{\tau _2}),\quad {H}^{\pm}_{\tau}= (H^{\pm}_{\tau _1}, H^{\pm}_{\tau _2}),\quad v^{\pm}_{\tau _i}=({v}^{\pm} ,{\tau}_i),
\]
\[
H^{\pm}_{\tau
_i}=({H}^{\pm} ,{\tau}_i),\quad
{\tau}_1=(\partial_2\varphi,1,0),\quad
{\tau}_2=(\partial_3\varphi,0,1),\quad H_{\rm N}|_{\Gamma}:=H_{\rm N}^{\pm}|_{\Gamma};
\]
$\mathfrak{j}:=\mathfrak{j}^{\pm}|_{\Gamma}$ is the mass transfer flux across the discontinuity surface.

From the mathematical point of view, there are two types of strong discontinuities: shock waves and characteristic discontinuities. Following Lax \cite{Lax57}, characteristic discontinuities, which are characteristic free boundaries, are called contact discontinuities. For the Euler equations of gas dynamics contact discontinuities are indeed contact  from the physical point of view, i.e., there is no flow across the discontinuity ($\mathfrak{j}=0$).

 In MHD the situation with characteristic discontinuities is richer than in gas dynamics. Namely, besides shock waves ($\mathfrak{j}\neq 0$, $[\rho ]\neq 0$) there are three types of characteristic discontinuities \cite{BThand,LL}: tangential discontinuities or current-vortex
sheets ($\mathfrak{j} =0$, $H_{\rm N}|_{\Gamma}= 0$), Alfv\'{e}n or rotational discontinuities ($\mathfrak{j}\neq 0$, $[\rho ] = 0$), and contact discontinuities ($\mathfrak{j}=0$, $H_{\rm N}|_{\Gamma} \neq 0$).
Current-vortex sheets and MHD contact discontinuities are contact from the physical point of view, but Alfv\'{e}n discontinuities are not.

Strong discontinuities formally introduced for the MHD equations do not necessarily exist (at least, locally in time) as piecewise smooth solutions for the full range of admissible initial flow parameters. As is known (see \cite{BS,BThand,Maj,Met}), the fulfilment of the Kreiss-Lopatinski condition \cite{Kreiss} for the linearized constant coefficients problem for a planar discontinuity is necessary (but not sufficient in general) for the local-in-time existence of corresponding nonplanar discontinuities. The violation of the Kreiss-Lopatinski condition is equivalent to the ill-posedness of the linearized constant coefficient problem. This is the same as {\it Kelvin-Helmholtz instability}, and the corresponding planar discontinuity is called unstable or violently unstable.

The more restrictive condition is the uniform Kreiss-Lopatinski condition \cite{Kreiss} and it requires the nonexistence of not only unstable but also neutral modes. The corresponding planar discontinuity is called {\it  uniformly stable}. The {uniform stability} condition satisfied at each point of the initial strong discontinuity usually implies its local-in-time existence. At least, this is so for gas dynamical shock waves \cite{BThand,Maj}, and more recently the same was proved for MHD shock waves \cite{MZ,Kwon}. However, {\it neutral (or weak) stability} sometimes also implies local-in-time existence. For example, for the isentropic Euler equations, Coulombel and Secchi \cite{CS2} have proved the local-in-time-existence of neutrally stable vortex sheets (in 2D) and neutrally stable shock waves.

In MHD there are two types of Lax shocks: fast and slow shock waves (see, e.g., \cite{LL}). A complete 2D stability analysis of fast MHD shock waves was carried out in \cite{T} for a polytropic gas equation of state. Taking into account the results of \cite{MZ,Kwon} extending the Kreiss--Majda theory \cite{Kreiss,Maj} to a class of hyperbolic symmetrizable systems with characteristics of variable multiplicities (this class contains the MHD system), uniformly stable fast shock waves found in \cite{T} exist locally in time. Regarding slow shock waves, some results about their stability  can be found in \cite{BThand,Fil}.

Current-vortex sheets can be only neutrally stable and a sufficient condition for the neutral stability of planar compressible current-vortex sheets\footnote{For the well-posedness and stability analysis for {\it incompressible} current-vortex sheets we refer the reader to \cite{CMST,MoTraTre,Ticvs} and references therein.} was found in \cite{T05}. The local-in-time existence of solutions with a surface of a current-vortex sheet of the MHD equations was proved in \cite{T09}, provided that the sufficient stability from \cite{T05} holds at each point of the initial discontinuity.

By numerical testing of the Kreiss-Lopatinski condition the parameter domains of stability and violent instability of planar Alfv\'{e}n discontinuities were found in \cite{IT}.  Alfv\'{e}n discontinuities are violently unstable for a wide range of flow parameters \cite{IT} but, as current-vortex sheets, they can be only neutrally stable. As was shown in \cite{IT}, the violation of uniform stability for stable Alfv\'{e}n discontinuities is the direct consequence of the fact that the symbol associated to the free boundary is {\it not elliptic}. It means that the boundary conditions cannot be resolved for the time-space gradient $\nabla_{t,x}\varphi=(\partial_t\varphi ,\partial_2\varphi ,\partial_3\varphi )$ of the front function.

In fact, for a general class of free boundary problems one can show that the non-ellipticity of the front symbol implies the existence of neutral modes associated with zeroing of this symbol. At the same time, neutral stability in the case of non-elliptic front symbol is, in some sense, weaker than usual neutral stability. As was noted in \cite{Tcpam}, this kind of neutral stability is, in general, not enough for the well-posedness of the linearized problem with variable coefficients (and, of course, of the original nonlinear problem). That is, this problem can be ill-posed even if the constant coefficients problem satisfies the (weak) Kreiss-Lopatinski condition.

The classical example of a free boundary problem with non-elliptic front symbol is the problem for the incompressible Euler equations with the vacuum boundary condition $p|_{\Gamma}=0$ on a free boundary $\Gamma (t)$ moving with the velocity of fluid particles (see \cite{Lind_incomp} and references therein). The front symbol is also not elliptic for the counterpart of this problem for the compressible Euler equations describing the motion of a compressible perfect liquid (with $\rho|_{\Gamma}>0$) in vacuum \cite{Lind,Tcpam}.

One can easily check that the linearization of the vacuum problem (for compressible or incompressible liquid) associated with a planar boundary always satisfies the Kreiss-Lopatinski condition, i.e., Kelvin-Helmholtz instability does not occur. Nevertheless, the local-in-time existence in Sobolev spaces was managed to be proved in \cite{Lind_incomp,Lind,Tcpam} for the nonlinear free boundary problem only under the physical assumption
\begin{equation}
\frac{\partial p}{\partial N}\leq -\epsilon <0\quad \mbox{on}\ \Gamma (0)
\label{vacc}
\end{equation}
on the normal derivative of the pressure on the initial free boundary, where $N$ is the outward normal to $\Gamma$ and $\epsilon$ is a fixed constant. As is known, the violation of \eqref{vacc} is associated with {\it Rayleigh-Taylor instability} occurring on the level of variable coefficients of the linearized problem. Moreover, for the case of incompressible liquid Ebin \cite{Ebin}
showed the ill-posedness in Sobolev spaces of the nonlinear problem  when the physical condition \eqref{vacc} is not satisfied.

In this paper, we are interested in MHD contact discontinuities and for them the front symbol is also always not elliptic. For contact discontinuities, in view of the requirements $\mathfrak{j}=0$ and $H_{\rm N}|_{\Gamma} \neq 0$, the Rankine-Hugoniot conditions \eqref{6} give the boundary conditions
\begin{equation}
[p]=0,\quad [v]=0,\quad [H]=0,\quad \partial_t\varphi =v^+_N\quad \mbox{on}\ \Gamma (t)
\label{bcond}
\end{equation}
which indeed cannot be resolved for $\nabla_{t,x}\varphi$.
At the same time, the density and the entropy may undergo any jump: $[\rho]\neq 0$, $[S]\neq 0$.

\begin{remark}
{\rm
Observe that the continuity of the magnetic field in \eqref{bcond} is equivalent to
\begin{equation}
[H_N]=0,\quad [H_{\tau}]=0\quad \mbox{on}\ \Gamma (t).
\label{bcondH}
\end{equation}
One can show that the first condition in \eqref{bcondH} coming from the constraint equation \eqref{2} is not a real boundary condition and must be regarded as a restriction (boundary constraint) on the initial data (see Proposition \ref{p1} for more details).
}
\label{r1}
\end{remark}

Our final goal is to find conditions on the initial data
\begin{equation}
{U}^{\pm} (0,{x})={U}_0^{\pm}({x}),\quad {x}\in \Omega^{\pm} (0),\quad \varphi (0,{x}')=\varphi _0({x}'),\quad {x}'\in\mathbb{R}^2,\label{indat}
\end{equation}
providing the existence and uniqueness in Sobolev spaces on some time interval $[0,T]$ of a solution $(U^{\pm},\varphi )$ to the free boundary problem \eqref{1}, \eqref{bcond}, \eqref{indat}, where $U^{\pm}:=U$ in $\Omega^{\pm}(t)$, and $U^{\pm}$ is smooth in $\Omega^{\pm}(t)$. Because of the general properties of hyperbolic conservation laws it is natural to expect only the local-in-time existence of solutions with a surface of contact discontinuity. Therefore, from the mathematical point of view, the question on its nonlinear Lyapunov's stability has no sense.

At the same time, the study of the linearized stability of contact discontinuities is not only a necessary step towards the proof of their local-in-time existence but is also of independent interest in connection with various astrophysical applications. As is noted in \cite{Goed}, the boundary conditions \eqref{bcond} are most typical for astrophysical plasmas. Contact discontinuities are usually observed in the solar wind, behind astrophysical shock waves bounding supernova remnants or due to the interaction of multiple shock waves driven by fast coronal mass ejections.

The absence of Kelvin-Helmholtz instability for contact discontinuities follows from a conserved energy integral \cite{BThand} which can be trivially obtained for the constant coefficients problem resulting from the linearization of problem \eqref{1}, \eqref{bcond}, \eqref{indat} about its piecewise constant solution associated with a planar discontinuity. That is, planar MHD contact discontinuities are always neutrally stable (see Section \ref{sec:3} for more details).\footnote{The surprising thing is that neutral stability trivially shown by the energy method cannot be proved by the usual spectral analysis due to the impossibility to solve analytically (in the general case) the dispersion relation for the magnetoacoustics system (see Section \ref{sec:3}).} But, the non-ellipticity of the front symbol causes principal difficulties if we want to extend the a priori $L^2$ estimate \cite{BThand} easily deduced for the constant coefficients problem to the linearized problem  with variable coefficients.

In this paper, under a suitable stability condition for the unperturbed flow (see just below) we manage to prove the well-posedness in Sobolev spaces of the linearized variable coefficients problem for contact discontinuities in 2D planar MHD (see \eqref{4}). It is amazing that the classical {\it Rayleigh-Taylor sign condition}
\begin{equation}
\left[\frac{\partial p}{\partial N}\right]\leq -\epsilon <0
\label{RT}
\end{equation}
on the jump of the normal derivative of the pressure satisfied at each point of the unperturbed nonplanar contact discontinuity naturally appears in our energy method as the condition sufficient for the well-posedness of the linearized problem. It is interesting to note that, unlike the condition $[\partial q/\partial N] <0$ considered in \cite{Tjde} for the plasma-vacuum interface problem, the magnetic field does not enter \eqref{RT}. That is, for MHD contact discontinuities condition \eqref{RT} appears in its classical (purely hydrodynamical) form as a condition for the pressure $p$ but not for the total pressure $q=p+\frac{1}{2}|{H} |^2$.

The well-posedness result of our paper is a necessary step to prove the local-in-time existence of MHD contact discontinuities provided that the Rayleigh-Taylor sign condition is satisfied at each point of the initial discontinuity. Since in the basic a priori estimate obtained in this paper for the variable coefficients linearized problem we have a {\it loss of derivatives} from the source terms to the solution, we plan to prove the existence of solutions to the original nonlinear problem by a suitable Nash-Moser-type iteration scheme. We suppose that the scheme of the proof will be really similar to that in \cite{CS2,T09,Tcpam}. We do not see any principal difficulties in this direction but postpone nonlinear analysis to a forthcoming paper.

The extension of our result to the general 3D case is still an open difficult problem. It is worth noting that the Rayleigh-Taylor instability of contact discontinuities was earlier detected in numerical MHD simulations of astrophysical plasmas as fingers near the contact discontinuity in the contour maps of density (see \cite{FangZhang} and references therein). Our hypothesis is that Rayleigh-Taylor instability is associated with the violation of the classical condition \eqref{RT}. The proof of this hypothesis is also an interesting open problem for future research.

The rest of the paper is organized as follows. In Section \ref{sec:2}, we reduce the free boundary problem \eqref{1}, \eqref{bcond}, \eqref{indat} to an initial-boundary value problem in a fixed domain and discuss properties of the reduced problem. In Section \ref{sec:3}, we obtain the linearized problem and formulate the main result for it which is Theorem \ref{t1} about its well-posedness in Sobolev spaces under the Rayleigh-Taylor sign condition for the 2D planar unperturbed flow. Moreover, in Section \ref{sec:3} we briefly discuss properties of the  constant coefficients linearized problem for a planar contact discontinuity. In Section \ref{sec:4} we derive the energy a priori estimate for the  variable coefficients linearized problem and in Section \ref{sec:exist} prove the existence of solutions of this problem.

\section{Reduced nonlinear problem in a fixed domain}
\label{sec:2}

The function $\varphi (t,x')$ determining the contact discontinuity surface $\Gamma$ is one of the unknowns of
the free boundary problem \eqref{1}/\eqref{3Dsys}, \eqref{bcond}, \eqref{indat}. To reduce this problem to that in a fixed domain we straighten the interface $\Gamma$ by using the same simplest change of independent variables as in \cite{T09,Tcpam}. That is, the unknowns $U^+$ and $U^-$ being smooth in $\Omega^{\pm}(t)$ are replaced by the vector-functions
\begin{equation}
\widetilde{U}^{\pm}(t,x ):= {U}^{\pm}(t,\Phi^{\pm} (t,x),x')
\label{change}
\end{equation}
which are smooth in the half-space $\mathbb{R}^3_+=\{x_1>0,\ x'\in \mathbb{R}^2\}$,
where
\begin{equation}
\Phi^{\pm}(t,x ):= \pm x_1+\Psi^{\pm}(t,x ),\quad \Psi^{\pm}(t,x ):= \chi (\pm x_1)\varphi (t,x'),
\label{change2}
\end{equation}
and $\chi\in C^{\infty}_0(\mathbb{R})$ equals to 1 on $[-1,1]$, and $\|\chi'\|_{L_{\infty}(\mathbb{R})}<1/2$. Here, as in \cite{Met}, we use the cut-off function $\chi$ to avoid assumptions about compact support of the initial data in our (future) nonlinear existence theorem. Alternatively, we could use the same change of variables as in \cite{CMST,MoTraTre-vacuum,SecTra,ST} inspired by Lannes \cite{Lannes} (see Remark \ref{r3} below).

\begin{remark}
{\rm
The above change of variable is admissible if $\partial_1\Phi^{\pm}\neq 0$. The latter is guaranteed, namely, the inequalities $\partial_1\Phi^+> 0$ and $\partial_1\Phi^-< 0$ are fulfilled, if we consider solutions for which $\|\varphi\|_{L_{\infty}([0,T]\times\mathbb{R}^2)}\leq 1$. This holds if,
without loss of generality, we consider the initial data satisfying $\|\varphi_0\|_{L_{\infty}(\mathbb{R}^2)}\leq 1/2$, and the time $T$ in our existence theorem is sufficiently small.}
\label{r2}
\end{remark}

\begin{remark}
{\rm
In the change of variables used in \cite{CMST,MoTraTre-vacuum,SecTra,ST} the function $\Psi^{\pm}$ in \eqref{change2} is defined in a different way. An important point of this change of variables is the regularization of one half of derivative of the lifting function $\Psi^{\pm}$ with respect to $\varphi$. This property was crucial for deriving a priori estimates for the nonlinear problem for incompressible current-vortex sheets in \cite{CMST}. At the same time, this change of variables gave no advantages for the linear analysis for the plasma-vacuum interface problem in \cite{MoTraTre-vacuum,SecTra}, and it just gave a gain of one half of derivative in the subsequent nonlinear existence theorem  in \cite{ST} proved by Nash-Moser iterations. This gain of one half of derivative in  \cite{ST} is only a nonprincipal technical improvement, and we expect the same for MHD contact discontinuities. At least, for the linearized problem for contact discontinuities, it does not matter whether we define the function $\Psi^{\pm}$ as in \eqref{change2} or as in \cite{CMST,MoTraTre-vacuum,SecTra,ST}. Therefore, choosing now the change of variables \eqref{change}, \eqref{change2}, we postpone the final choice to the forthcoming nonlinear analysis.}
\label{r3}
\end{remark}

Dropping for convenience tildes in $\widetilde{U}^{\pm}$, we reduce \eqref{3Dsys}, \eqref{bcond}, \eqref{indat} to the initial-boundary value problem
\begin{equation}
\mathbb{L}(U^+,\Psi^+)=0,\quad \mathbb{L}(U^-,\Psi^-)=0\quad\mbox{in}\ [0,T]\times \mathbb{R}^3_+,\label{11}
\end{equation}
\begin{equation}
\mathbb{B}(U^+,U^-,\varphi )=0\quad\mbox{on}\ [0,T]\times\{x_1=0\}\times\mathbb{R}^{2},\label{12}
\end{equation}
\begin{equation}
U^+|_{t=0}=U^+_0,\quad U^-|_{t=0}=U^-_0\quad\mbox{in}\ \mathbb{R}^3_+,
\qquad \varphi |_{t=0}=\varphi_0\quad \mbox{in}\ \mathbb{R}^{2},\label{13}
\end{equation}
where $\mathbb{L}(U,\Psi)=L(U,\Psi)U$,
\[
L(U,\Psi)=A_0(U)\partial_t +\widetilde{A}_1(U,\Psi)\partial_1+A_2(U )\partial_2+A_3(U )\partial_3,
\]
\[
\widetilde{A}_1(U^{\pm},\Psi^{\pm})=\frac{1}{\partial_1\Phi^{\pm}}\Bigl(
A_1(U^{\pm})-A_0(U^{\pm})\partial_t\Psi^{\pm}-\sum_{k=2}^3A_k(U^{\pm})\partial_k\Psi^{\pm}\Bigr)
\]
($\partial_1\Phi^{\pm}=\pm 1 +\partial_1\Psi^{\pm}$), and (\ref{12}) is the compact form of the boundary conditions
\begin{equation}
[p]=0,\quad [v]=0,\quad [H_{\tau}]=0,\quad \partial_t\varphi-v_{N}^+|_{x_1=0}=0,
\label{12'}
\end{equation}
with $[g]:=g^+|_{x_1=0}-g^-|_{x_1=0}$ for any pair of values $g^+$ and $g^-$. Moreover, recall that, according to the definition of contact discontinuity, $H^{\pm}_N|_{x_1=0}\neq 0$. Here
\[
v_{N}^{\pm}=v_1^{\pm}- v_2^{\pm}\partial_2\Psi^{\pm}- v_3^{\pm}\partial_3\Psi^{\pm},\quad H_{N}^{\pm}=H_1^{\pm}-H_2^{\pm}\partial_2\Psi^{\pm}-H_3^{\pm}\partial_3\Psi^{\pm},
\]
\[
{H}^{\pm}_{\tau}= (H^{\pm}_{\tau _1}, H^{\pm}_{\tau _2}),\quad H^{\pm}_{\tau
_i}=H_1^\pm \partial_{i+1}\Psi^{\pm}+H_{i+1}^{\pm}, \quad i=1,2.
\]

There appear two natural questions. The first one: Why have the boundary condition $[H_N]=0$ not been included in \eqref{12} (see Remark \ref{r1})? And the second question: Why are systems (\ref{1}) and
\eqref{3Dsys} equivalent on solutions with a surface of contact discontinuity, i.e., why is system (\ref{1}) in the straightened variables equivalent to (\ref{11})? The answer to these questions is given by the following proposition.

\begin{proposition}
Let the initial data \eqref{13} satisfy
\begin{equation}
{\rm div}\, h^+=0,\quad {\rm div}\, h^-=0
\label{14}	
\end{equation}
and the boundary condition
\begin{equation}
[H_{N}]=0,
\label{15}
\end{equation}
where $h^{\pm}=(H_{N}^{\pm},H_2^{\pm}\partial_1\Phi^{\pm},H_3^{\pm}\partial_1\Phi^{\pm})$. If problem \eqref{11}--\eqref{13} has a sufficiently smooth solution, then this solution satisfies \eqref{14} and \eqref{15} for all $t\in [0,T]$. The same is true for solutions  with a surface of contact discontinuity of system \eqref{1}.
\label{p1}
\end{proposition}

\begin{proof}
The proof that equations \eqref{14} are satisfied for all $t\in [0,T]$ if they are true at $t=0$ is absolutely the same as in \cite{T09}. Regarding the boundary condition \eqref{15}, again, exactly as in Appendix A in \cite{T09}, we consider the equations for $H^{\pm}$ contained in \eqref{11} on the boundary $x_1=0$, we use  the last condition in \eqref{12'} and its counterpart for $v^-$  to obtain
\[
\partial_t H_{N}^{\pm} +v_2^{\pm}\partial_2H_{N}^{\pm}+v_3^{\pm}\partial_3H_{N}^{\pm}+
\left(\partial_2v_2^{\pm}+\partial_3v_3^{\pm}\right) H_{N}^{\pm}=0\quad \mbox{on}\ x_1=0.
\]
In view of $[v]=0$, the last equations imply
\[
\partial_t [H_{N}] +v_2^+\partial_2[H_{N}]+v_3^+\partial_3[H_{N}]+ \left(\partial_2v_2^++\partial_3v_3^+\right) [H_{N}]=0\quad \mbox{on}\ x_1=0.
\]
Using the standard method of characteristic curves, we conclude that (\ref{15}) is fulfilled for all $t\in [0,T]$ if it is satisfied for $t=0$.
\end{proof}

Equations (\ref{14}) are just the divergence constraint (\ref{2}) on either side of the straightened front. Using (\ref{14}), we can prove that system (\ref{1}) in the straightened variables is equivalent to (\ref{11}).

Concerning the boundary condition \eqref{15}, we must regard it as the restriction on the initial data (\ref{13}). Otherwise, the hyperbolic problem (\ref{11}), (\ref{12}), (\ref{15}) does not have a correct number of boundary conditions.
Indeed, the boundary matrix reads
\[
A_{\nu}={\rm diag}\, \bigl(\widetilde{A}_1({\bf U}^+,\Psi^+),\widetilde{A}_1({\bf U}^-,\Psi^-)\bigr),
\]
where
\[
\widetilde{A}_1(U^{\pm},\Psi^{\pm}) = \frac{1}{\partial_1 \Phi^{\pm}}\begin{pmatrix} \frac{w_1^{\pm}}{\rho^{\pm}(c^{\pm})^2} & N^{\pm} & 0& 0\\[6pt]
(N^{\pm})^T & \rho w_1^{\pm} I_3 & N^{\pm} \otimes H^{\pm} - h_1^{\pm} I_3 & 0^T\\[6pt]
0^T & H^{\pm} \otimes N^{\pm} - h_1^{\pm} I_3 & w_1^{\pm} I_3 & 0^T\\[3pt]
0& 0& 0 & w_1^{\pm}
\end{pmatrix},
\]
\[
w^{\pm}=u^{\pm}-(\partial_t\Psi^{\pm},0,0),\quad u^{\pm}=(v_N^{\pm},v_2^{\pm}\partial_1\Phi^{\pm},v_3^{\pm}\partial_1\Phi^{\pm}),
\]
$I_3$ is the unit matrix of order 3, $N^{\pm}=(1, -\partial_2\Psi^{\pm}, -\partial_3\Psi^{\pm})$, and $w_1^{\pm}$ and $h_1^{\pm}$ are the first components of the vectors $w^{\pm}$ and $h^{\pm}$ respectively, and $\otimes$ denotes the tensor product. In view of the last condition in \eqref{12'} and its counterpart for $v^-$, we have $w_1^{\pm}|_{x_1=0}=0$ and
\begin{equation}
\widetilde{A}_1(U^{\pm},\Psi^{\pm})|_{x_1=0} = \pm \begin{pmatrix} 0 & N & 0& 0\\[6pt]
N^T & O_3 & N \otimes H^{\pm} - h_1^{\pm} I_3 & 0^T\\[6pt]
0^T & H^{\pm} \otimes N - h_1^{\pm} I_3 & O_3 & 0^T\\[3pt]
0& 0& 0 & 0
\end{pmatrix}_{|x_1=0},
\label{A1tilde}
\end{equation}
where $O_3$ is the zero matrix of order 3. For the matrix
\[
\widehat{\mathcal{A}}^{\pm}=\widetilde{A}_1(\widehat{U}^{\pm},\hat{\Psi}^{\pm})|_{x_1=0}
\]
calculated on a certain background $(\widehat{U}^{\pm},\hat{\varphi} )$ we have
\begin{equation}
(\widehat{\mathcal{A}}^{\pm}U^{\pm},U^{\pm})=((J^{\pm})^T\widehat{\mathcal{A}}^{\pm}J^{\pm}W^{\pm},W^{\pm})=(\widehat{\mathcal{B}}^{\pm}W^{\pm},W^{\pm}),
\label{boundmatr}
\end{equation}
where $\widehat{\mathcal{B}}^{\pm}=\mathcal{B}^{\pm}(\widehat{U}^{\pm}_{|x_1=0},\hat{\varphi} )$,
\[
W^{\pm}= (\check{q}^{\pm},\check{v}^{\pm}_N,v_2^{\pm},v_3^{\pm},\check{H}^{\pm}_N,H^{\pm}_2,H^{\pm}_3,S^{\pm}),\quad \check{q}^{\pm}=p^{\pm}+(\widehat{H}^{\pm},H^{\pm}),
\]
\[
 \check{v}^{\pm}_N=v_1^{\pm}-v_2^{\pm}\partial_2\hat{\varphi}-v_3^{\pm}\partial_3\hat{\varphi},\quad
\check{H}^{\pm}_N=H_1^{\pm}-H_2^{\pm}\partial_2\hat{\varphi}-H_3^{\pm}\partial_3\hat{\varphi},
\]
and $U^{\pm}=J^{\pm}W^{\pm}$, with $\det J^{\pm}\neq 0$.

After calculations analogous to those in \cite{SecTra} we find
\begin{equation}
\mathcal{B}^{\pm}=\pm\begin{pmatrix}
0 & e_1 & 0 & 0 \\[3pt]
e_1^T &O_3& -h_1^{\pm}a_0 & 0^T\\[3pt]
0^T & -h_1^{\pm}a_0 & O_3& 0^T\\[3pt]
0 & 0 &0 & 0
\end{pmatrix}_{|x_1=0},
\label{boundmatr'}
\end{equation}
where $e_1= (1,0,0)$ and $a_0=a_0(\varphi )$ is the symmetric positive definite matrix\footnote{The matrix $a_0$ is the matrix $\hat{a}_0$ from \cite{SecTra} calculated on the boundary $x_1=0$.}
\[
a_0=\begin{pmatrix}
1 & \partial_2\varphi &  \partial_3\varphi \\[3pt]
\partial_2\varphi & 1+(\partial_2\varphi)^2 &\partial_2\varphi\partial_3\varphi\\[3pt]
\partial_3\varphi &\partial_2\varphi\partial_3\varphi & 1+(\partial_3\varphi)^2
\end{pmatrix}.
\]
Since $h_1^{\pm}|_{x_1=0}\neq 0$, the matrix $\mathcal{B}^{\pm}$ has tree positive, three negative and two zero eigenvalues. Therefore, the boundary matrix $A_{\nu}$ on the boundary $x_1=0$ has six positive, six negative and four zero eigenvalues. That is, the boundary $x_1=0$ is {\it characteristic}, and since one of the boundary conditions is needed for determining the function $\varphi $, the correct number of boundary conditions is seven (that is the case in \eqref{12'}).

\begin{remark}
{\rm
In fact, the signature of the symmetric matrix $A_{\nu}|_{x_1=0}$  associated with a nonplanar front $x_1=\varphi (t,x')$ is determined from the well-known formulae for the eigenvalues of the matrix $A_1$ associated with the planar front $x_1=0$  and calculated through the Alfv\'{e}n and (fast and slow) magnetosonic velocities (see, e.g., \cite{LL}). In this sense, the calculations above were not really necessary. At the same time, for the linearized problem we will need formulae \eqref{boundmatr} and \eqref{boundmatr'} for deriving a priori estimates by the energy method.}
\label{r4}
\end{remark}

In the next section, we linearize problem \eqref{11}--\eqref{13} around a given basic state (``unperturbed flow") and we will have to make reasonable assumptions about it. Since in our future nonlinear analysis by Nash-Moser iterations the basic state playing the role of an intermediate state $(U^{\pm}_{n+1/2},\varphi_{n+1/2})$ (see \cite{CS2,ST,T09,Tcpam}) should finally converge to a solution of the nonlinear problem, we need to know a certain a priori information about solutions of problem \eqref{11}--\eqref{13}. This information is contained in the following proposition.

\begin{proposition}
Assume that problem \eqref{11}--\eqref{13} (with the initial data satisfying \eqref{14} and \eqref{15}) has a sufficiently smooth solution $(U^{\pm},\varphi)$ on a time interval $[0,T]$.  Assume also that the plasma obeys the state equation \eqref{polgas} of a
polytropic gas. Then the normal derivatives $\partial_1U^{\pm}$ satisfy the jump conditions
\begin{equation}
[\partial_1H_{N}]=0
\label{jc1}
\end{equation}
and
\begin{equation}
[\partial_1v]=0
\label{jc2}
\end{equation}
\label{p2}
for all $t\in [0,T]$, where due to the fact that we have transformed the domains $\Omega^\pm (t)$ into the same half-space $\mathbb{R}^3_+$ (but not into the different half-spaces $\mathbb{R}^3_+$ and $\mathbb{R}^3_-$) the jump of a normal derivative is defined as follows
\begin{equation}
[\partial_1a]:=\partial_1a^+_{|x_1=0}+\partial_1a^-_{|x_1=0}.
\label{norm_jump}
\end{equation}
\end{proposition}

\begin{proof}
It follows from \eqref{14} that ${\rm div}\, h^+_{|x_1=0}+{\rm div}\, h^-_{|x_1=0}=0$. Then, since $[H]=0$, we get \eqref{jc1}. For a polytropic gas, restricting equations \eqref{11} from $\mathbb{R}^3_+$ to the boundary $x_1=0$ and using the last condition in \eqref{12'} and its counterpart for $v^-$, we obtain
\[
\frac{1}{\gamma p^\pm }\partial^{\pm}_0p^\pm \pm {\rm div}\,u^\pm=0\quad\mbox{on}\ x_1=0,
\]
where $\partial^{\pm}_0:=\partial_t+v_2^\pm\partial_2 +v_3^\pm\partial_3$.
Passing to the jump and using the continuity of the pressure and the velocity, we get  ${\rm div}\, u^+_{|x_1=0}+{\rm div}\, u^-_{|x_1=0}=0$, and then
\begin{equation}
[\partial_1v_N]=0.
\label{vn}
\end{equation}
Again considering equations \eqref{11} on the boundary $x_1=0$ and passing to the jump, we obtain
\[
[\partial_0 H]-H^+_N[\partial_1 v]  -H^+_2[\partial_2 v]-H^+_3[\partial_3 v]+H^+({\rm div}\, u^+ +{\rm div}\, u^-)=0\quad\mbox{on}\ x_1=0,
\]
i.e.,
\[
H^+_N[\partial_1 v]=H^+({\rm div}\, u^+ +{\rm div}\, u^-) =H^+[\partial_1v_N]\quad\mbox{on}\ x_1=0.
\]
In view  of \eqref{vn} and the condition $H_N^+|_{x_1=0}\neq 0$, this gives \eqref{jc2}.
\end{proof}

\section{Linearized problem and main result}
\label{sec:3}

\subsection{Basic state}

We first describe a basic state upon which we perform linearization. Let
\begin{equation}
(\widehat{U}^+(t,x ),\widehat{U}^-(t,x ),\hat{\varphi}(t,{x}'))
\label{a21}
\end{equation}
be a given sufficiently smooth vector-function with $\widehat{U}^{\pm}=(\hat{p}^{\pm},\hat{v}^{\pm},\widehat{H}^{\pm},\widehat{S}^{\pm})$ and
\begin{equation}
\|\widehat{U}^+\|_{W^2_{\infty}(\Omega_T)}+
\|\widehat{U}^-\|_{W^2_{\infty}(\Omega_T)}
+\|\hat{\varphi}\|_{W^3_{\infty}(\partial\Omega_T)} \leq K,
\label{a22}
\end{equation}
where $K>0$ is a constant and
\[
\Omega_T:= (-\infty, T]\times\mathbb{R}^3_+,\quad \partial\Omega_T:=(-\infty ,T]\times\{x_1=0\}\times\mathbb{R}^{2}.
\]
If the basic state \eqref{a21} upon which we shall linearize problem \eqref{11}--\eqref{13} is a solution of this problem (its existence should be proved), then it is natural to call it unperturbed flow. The trivial example of the unperturbed flow is the constant solution $(\overline{U}^+,\overline{U}^-,0)$ associated with the planar contact discontinuity $x_1=0$, where $\overline{U}^{\pm}\in \mathbb{R}^{8}$  are constant vectors.

Considering from now on the case of a polytropic gas, we assume that the basic state defined in ${\Omega_T}$ satisfies there the hyperbolicity condition \eqref{5},
\begin{equation}
\rho (\hat{p}^{\pm},\widehat{S}^{\pm})\geq \bar{\rho}_0 >0,\quad \hat{p}^{\pm} \geq \bar{p}_0 >0  \label{a5}
\end{equation}
(with some fixed constants $\bar{\rho}_0$ and $\bar{p}_0$), the boundary conditions \eqref{12'},
\begin{equation}
[\hat{p}]=0,\quad [\hat{v}]=0,\quad [\widehat{H}_{\tau}]=0, \quad \partial_t\hat{\varphi}-\hat{v}_{N}^+|_{x_1=0}=0,
\label{a12'}
\end{equation}
the condition
\begin{equation}
|\widehat{H}_N^{\pm}|_{x_1=0}|\geq \bar{\kappa}_0 >0
\label{cdass}
\end{equation}
(with some fixed constant $\bar{\kappa}_0$), the equations for $H^{\pm}$ contained in \eqref{11},
\begin{equation}
\partial_t\widehat{H}^{\pm}+\frac{1}{\partial_1\widehat{\Phi}^{\pm}}\left\{ (\hat{w}^{\pm} ,\nabla )
\widehat{H}^{\pm} - (\hat{h}^{\pm} ,\nabla ) \hat{v}^{\pm} + \widehat{H}^{\pm}{\rm div}\,\hat{u}^{\pm}\right\} =0 ,
\label{b21}	
\end{equation}
the divergence constraints \eqref{14} and the boundary constraint \eqref{15} for $t\leq 0$,
\begin{equation}
{\rm div}\, \hat{h}^{\pm}|_{t\leq 0}=0,\quad [\widehat{H}_N]|_{t\leq 0}=0,
\label{b14}	
\end{equation}
and the jump condition \eqref{vn},
\begin{equation}
[\partial_1\hat{v}_N]=0,
\label{avn}
\end{equation}
where, from now on, the jump of the normal derivatives has to be intended as in \eqref{norm_jump} and
where all of the ``hat'' values are determined like corresponding values for $(U^\pm, \varphi)$, e.g.,
\[
\widehat{\Phi}^{\pm}(t,x )=\pm x_1 +\widehat{\Psi}^{\pm}(t,x ),\quad
\widehat{\Psi}^{\pm}(t,x )=\chi(\pm x_1)\hat{\varphi}(t,x'),
\]
\[
\hat{v}_{N}^{\pm}=\hat{v}_1^{\pm}- \hat{v}_2^{\pm}\partial_2\widehat{\Psi}^{\pm}- \hat{v}_3^{\pm}\partial_3\widehat{\Psi}^{\pm},\quad
\widehat{H}_{N}^{\pm}=\widehat{H}_1^{\pm}- \widehat{H}_2^{\pm}\partial_2\widehat{\Psi}^{\pm}- \widehat{H}_3^{\pm}\partial_3\widehat{\Psi}^{\pm},
\]
\[
\widehat{H}^{\pm}_{\tau}= (\widehat H^{\pm}_{\tau _1}, \widehat H^{\pm}_{\tau _2}),\quad \widehat H^{\pm}_{\tau
_i}=\widehat H_1^\pm \partial_{i+1}\Psi^{\pm}+\widehat H^{\pm}_{i+1}, \quad i=1,2.
\]
Moreover, without loss of generality we assume that $\|\hat{\varphi}\|_{L_{\infty}(\partial\Omega_T)}<1$ (see Remark \ref{r2}). This implies
\[
\partial_1\widehat{\Phi}^+\geq 1/2,\quad \partial_1\widehat{\Phi}^-\leq -  1/2.
\]
Note that \eqref{a22} yields
\[
\|\widehat{W}\|_{W^2_{\infty}(\Omega_T)} \leq C(K),
\]
where $\widehat{W}:=(\widehat{U}^+,\widehat{U}^-,\nabla_{t,x}\widehat{\Psi}^+,
\nabla_{t,x}\widehat{\Psi}^-)$, $\nabla_{t,x}=(\partial_t, \nabla )$, and $C=C(K)>0$ is a constant depending on $K$.

As in Proposition \ref{p1}, equations \eqref{b21} and constraints \eqref{b14} imply
\begin{equation}
{\rm div}\, \hat{h}^+=0, \quad {\rm div}\, \hat{h}^-=0
\label{b14'}	
\end{equation}
and
\begin{equation}
[\widehat{H}_N]=0
\label{b15}	
\end{equation}
for all $t\in (-\infty ,T]$. Then, as in the proof of Proposition \ref{p2}, using \eqref{a12'}, \eqref{avn} and \eqref{b15}, from \eqref{b21} and \eqref{b14'} we deduce for the basic state the jump conditions \eqref{jc1} and \eqref{jc2},
\begin{equation}
[\partial_1\widehat{H}_{N}]=0,\quad [\partial_1\hat{v}]=0.
\label{jc1'}
\end{equation}

\begin{remark}
{\rm
For the linearized problem we will need equations associated to the divergence constraints \eqref{14}. However, exactly as in \cite{T09}, to deduce them it is not enough that these constraints are satisfied by the basic state (\ref{a21}) (see \eqref{b14'}) and  we need actually that the equations for $H^\pm$ themselves contained in \eqref{11}  are fulfilled for (\ref{a21}), i.e., assumption \eqref{b21} holds.}
\label{r5}
\end{remark}

\begin{remark}
{\rm
Assumptions \eqref{a5}--\eqref{avn} are nonlinear constraints on the basic state which are automatically satisfied if the basic state is an exact solution of problem \eqref{11}--\eqref{13} (unperturbed flow). We will really need them while deriving a priori estimates for the linearized problem. In the forthcoming nonlinear analysis we plan to use the Nash-Moser method. As in \cite{CS2,T09,Tcpam}, the Nash-Moser procedure will be not completely standard. Namely, at each $n$th Nash-Moser iteration step we will have to construct an intermediate state $(U^{\pm}_{n+1/2},\varphi_{n+1/2})$ satisfying constraints  \eqref{a5}--\eqref{avn}. Without assumption \eqref{avn} such an intermediate state can be constructed in exactly the same manner as in \cite{T09}. Assumption \eqref{avn} does not however cause additional difficulties and we postpone corresponding arguments to the nonlinear analysis.}
\label{r6}
\end{remark}

\subsection{Linearized problem}

The linearized equations for (\ref{11}), (\ref{12}) read:
\[
\mathbb{L}'(\widehat{U}^{\pm},\widehat{\Psi}^{\pm})(\delta U^{\pm},\delta\Psi^{\pm}):=
\frac{d}{d\varepsilon}\mathbb{L}(U_{\varepsilon}^{\pm},\Psi_{\varepsilon}^{\pm})|_{\varepsilon =0}={f}^{\pm}
\quad \mbox{in}\ \Omega_T,
\]
\[
\mathbb{B}'(\widehat{U}^+,\widehat{U}^-,\hat{\varphi})(\delta U^+,\delta U^-,\delta \varphi):=
\frac{d}{d\varepsilon}\mathbb{B}(U_{\varepsilon}^+,U_{\varepsilon}^-,\varphi_{\varepsilon})|_{\varepsilon =0}={g}
\quad \mbox{on}\ \partial\Omega_T
\]
where $U_{\varepsilon}^{\pm}=\widehat{U}^{\pm}+ \varepsilon\,\delta U^{\pm}$,
$\varphi_{\varepsilon}=\hat{\varphi}+ \varepsilon\,\delta \varphi$, and
\[
\Psi_{\varepsilon}^{\pm}(t,{x} ):=\chi (\pm x_1)\varphi _{\varepsilon}(t,{x}'),\quad
\Phi_{\varepsilon}^{\pm}(t,{x} ):=\pm x_1+\Psi_{\varepsilon}^{\pm}(t,{x} ),
\]
\[
\delta\Psi^{\pm}(t,{x} ):=\chi (\pm x_1)\delta \varphi (t,{x} ).
\]
Here, as usual, we introduce the source terms ${f}^{\pm}(t,{x} )=(f_1^{\pm}(t,{x} ),\ldots ,f_8^{\pm}(t,{x} ))$ and ${g}(t,{x}' )=(g_1(t,{x}' ),\ldots ,g_7(t,{x}' ))$ to make the interior equations and the boundary conditions inhomogeneous.

We easily compute the exact form of the linearized equations (below we drop $\delta$):
\[
\mathbb{L}'(\widehat{{U}}^{\pm},\widehat{\Psi}^{\pm})({U}^{\pm},\Psi^{\pm})\\
=
L(\widehat{{U}}^{\pm},\widehat{\Psi}^{\pm}){U}^{\pm} +{\cal C}(\widehat{{U}}^{\pm},\widehat{\Psi}^{\pm})
{U}^{\pm} -  \bigl\{L(\widehat{{U}}^{\pm},\widehat{\Psi}^{\pm})\Psi^{\pm}\bigr\}\frac{\partial_1\widehat{U}{\pm}}{\partial_1\widehat{\Phi}^{\pm}},
\]
\[
\mathbb{B}'(\widehat{{U}}^+,\widehat{{U}}^-,\hat{\varphi})({U}^+,{U}^-,f)=
\left(
\begin{array}{c}
p^+-p^-\\[3pt]
v^+-v^-\\[3pt]
H_{\tau}^+-H_{\tau}^-\\[3pt]
\partial_t\varphi +\hat{v}_2^+\partial_2\varphi +\hat{v}_3^+\partial_3\varphi -v_{N}^+
\end{array}
\right),
\]
where
\[
v_{N}^{\pm}=
v_1^{\pm}-v_2^{\pm}\partial_2\widehat{\Psi}^\pm -v_3^{\pm}\partial_3\widehat{\Psi}^\pm , \quad
{H}^{\pm}_{\tau}= (H^{\pm}_{\tau _1}, H^{\pm}_{\tau _2}),
\]
\[
 H^{\pm}_{\tau
_i}=H_1^\pm \partial_{i+1}\widehat{\Psi}^{\pm}+H^{\pm}_{i+1}, \quad i=1,2,
\]
and the matrix
${\cal C}(\widehat{{U}}^{\pm},\widehat{\Psi}^{\pm})$ is determined as follows:
\begin{multline*}
{\cal C}(\widehat{{U}}^{\pm},\widehat{\Psi}^{\pm}){Y}
= ({Y} ,\nabla_yA_0(\widehat{{U}}^{\pm} ))\partial_t\widehat{{U}}^{\pm}
 +({Y} ,\nabla_y\widetilde{A}_1(\widehat{\bf U}^{\pm},\widehat{\Psi}^{\pm}))\partial_1\widehat{{U}}^{\pm}\\[6pt]
+ ({Y} ,\nabla_yA_2(\widehat{{U}}^{\pm} ))\partial_2\widehat{{U}}^{\pm}
+ ({Y} ,\nabla_yA_3(\widehat{{U}}^{\pm} ))\partial_3\widehat{{U}}^{\pm},
\end{multline*}
\[
({Y} ,\nabla_y A(\widehat{{U}}^{\pm})):=\sum_{i=1}^8y_i\left.\left(\frac{\partial A ({Y} )}{
\partial y_i}\right|_{{Y} =\widehat{{U}}^{\pm}}\right),\quad {Y} =(y_1,\ldots ,y_8).
\]

The differential operator $\mathbb{L}'(\widehat{U}^{\pm},\widehat{\Psi}^{\pm})$ is a first order operator in
$\Psi^{\pm}$. This fact can give some trouble in obtaining a priori estimates for the linearized problem by the energy method. Following
\cite{Al}, we overcome this difficulty by introducing the ``good unknowns'':
\begin{equation}
\dot{U}^+:=U^+ -\frac{\Psi^+}{\partial_1\widehat{\Phi}^+}\,\partial_1\widehat{U}^+,\quad
\dot{U}^-:=U^- -\frac{\Psi^-}{\partial_1\widehat{\Phi}^-}\,\partial_1\widehat{U}^- .
\label{b23}
\end{equation}
Omitting detailed calculations, we rewrite the linearized interior equations in terms of the new unknowns \eqref{b23}:
\begin{equation}
L(\widehat{U}^{\pm},\widehat{\Psi}^{\pm})\dot{U}^{\pm} +{\cal C}(\widehat{U}^{\pm},\widehat{\Psi}^{\pm})
\dot{U}^{\pm} - \frac{\Psi^{\pm}}{\partial_1\widehat{\Phi}^{\pm}}\,\partial_1\bigl\{\mathbb{L}
(\widehat{U}^{\pm},\widehat{\Psi}^{\pm})\bigr\}={f}^{\pm}.
\label{b24}
\end{equation}
Dropping as in \cite{Al,CS2,MoTraTre-vacuum,SecTra,ST,T09,Tcpam,Tjde} the zero-order terms in $\Psi^+$  and $\Psi^-$ in \eqref{b24},\footnote{In the future nonlinear analysis the dropped terms in \eqref{b24} should be considered as error terms at each Nash-Moser iteration step.} we write down the final form of our linearized problem for $(\dot{U}^+,\dot{U}^-,\varphi )$:
\begin{equation}
 L(\widehat{U}^{\pm},\widehat{\Psi}^{\pm})\dot{U}^{\pm} +\mathcal{C}(\widehat{U}^{\pm},\widehat{\Psi}^{\pm})
\dot{U}^{\pm} =f^{\pm}\qquad \mbox{in}\ \Omega_T,\label{b48}
\end{equation}
\begin{equation}
\left(
\begin{array}{c}
\dot{p}^+-\dot{p}^- + \varphi [\partial_1\hat{p}]\\[3pt]
\dot{v}^+-\dot{v}^-\\
\dot{H}_{\tau}^+-\dot{H}_{\tau}^-+ \varphi [\partial_1\widehat{H}_{\tau}]\\[3pt]
\partial_t\varphi +\hat{v}_2^+\partial_2\varphi +\hat{v}_3^+\partial_3\varphi -\dot{v}_{N}^+ - \varphi \partial_1\hat{v}_N^+
\end{array}
\right)=g \qquad \mbox{on}\ \partial\Omega_T,
\label{b50}
\end{equation}
\begin{equation}
(\dot{U}^+,\dot{U}^-,\varphi )=0\qquad \mbox{for}\ t<0,\label{b51}
\end{equation}
where
\[
\dot{v}_{N}^{\pm}=
\dot{v}_1^{\pm}-\dot{v}_2^{\pm}\partial_2\widehat{\Psi}^\pm -\dot{v}_3^{\pm}\partial_3\widehat{\Psi}^\pm , \quad
\dot{H}^{\pm}_{\tau}= (\dot{H}^{\pm}_{\tau_1}, \dot{H}^{\pm}_{\tau_2}),
\]
\[
 \dot{H}^{\pm}_{\tau_i}=\dot{H}_1^\pm \partial_{i+1}\widehat{\Psi}^{\pm}+\dot{H}^{\pm}_{i+1}, \quad i=1,2.
\]
We used the important condition $[\partial_1\hat{v}]=0$ for the basic state, cf. \eqref{jc1'}, while writing down the second line in the left-hand side of the boundary conditions \eqref{b50}. We assume that $f^{\pm}$ and $g$ vanish in the past and consider the case of zero initial data, which is the usual assumption.\footnote{The case of nonzero initial data is postponed to the nonlinear analysis (construction of a so-called approximate solution; see, e.g., \cite{CS2,T09}).}

From problem \eqref{b48}--\eqref{b51} we can deduce nonhomogeneous equations associated with the divergence constraints \eqref{14} and the ``redundant'' boundary condition \eqref{15} for the nonlinear problem. More precisely, we have the following.

\begin{proposition}
Let the basic state \eqref{a21} satisfies assumptions \eqref{a5}--\eqref{avn}.
Then solutions of problem \eqref{b48}--\eqref{b51} satisfy
\begin{equation}
{\rm div}\,\dot{h}^+=f_9^+,\quad {\rm div}\,\dot{h}^-=f_9^-\quad\mbox{in}\ \Omega_T,
\label{b43}
\end{equation}
\begin{equation}
\dot{H}_{N}^+-\dot{H}_N^-=g_8\quad\mbox{on}\ \partial\Omega_T.
\label{b44}
\end{equation}
Here
\[
\dot{h}^{\pm}=(\dot{H}_{N}^{\pm},\dot{H}_2^{\pm}\partial_1\widehat{\Phi}^{\pm},\dot{H}_3^{\pm}\partial_1\widehat{\Phi}^{\pm}),\quad
\dot{H}_{N}^{\pm}=\dot{H}_1^{\pm}-\dot{H}_2^{\pm}\partial_2\widehat{\Psi}^{\pm}-\dot{H}_3^{\pm}
\partial_3\widehat{\Psi}^{\pm}
\]
and the functions $f_9^{\pm}=f_9^{\pm}(t,x )$ and $g_8= g_8(t,x')$, which vanish in the past, are determined by the source terms and the basic state as solutions to the linear inhomogeneous equations
\begin{equation}
\partial_t a^{\pm}+ \frac{1}{\partial_1\widehat{\Phi}^{\pm}}\left\{ (\hat{w}^{\pm} ,\nabla a^{\pm}) + a^{\pm}\,{\rm div}\,\hat{u}^{\pm}\right\}={\mathcal F}^{\pm}\quad \mbox{in}\ \Omega_T,
\label{aa1}
\end{equation}
\begin{equation}
\partial_t g_8 +\partial_2(\hat{v}_2^+g_8)+\partial_3(\hat{v}_3^+g_8)
={\mathcal G}\quad \mbox{on}\ \partial\Omega_T,
\label{aa2}
\end{equation}
where $a^{\pm}=f_9^{\pm}/\partial_1\widehat{\Phi}^{\pm},\quad {\mathcal F}^{\pm}=({\rm div}\,{f}_{h}^{\pm})/\partial_1\widehat{\Phi}^{\pm}$,
\[
{f}_{h}^{\pm}=(f_{N}^{\pm} ,\partial_1\widehat{\Phi}^{\pm}f_6^{\pm},\partial_1\widehat{\Phi}^{\pm}f_7^{\pm}),\quad f_{N}^{\pm}=f_5^{\pm}-f_6^{\pm}\partial_2\widehat{\Psi}^{\pm}-
f_7^{\pm}\partial_3\widehat{\Psi}^{\pm},
\]
\[
{\mathcal G}=\left\{\left.[f_{N}]+\partial_2\bigl(\widehat{H}^+_2g_N\bigr)+\partial_3\bigl(\widehat{H}^+_3g_N\bigr)-\partial_2\bigl(\widehat{H}^+_Ng_3\bigr)
-\partial_3\bigl(\widehat{H}^+_Ng_4\bigr)\right\}\right|_{x_1=0},
\]
\[
[f_{N}]=f_{N}^+|_{x_1=0}-f_{N}^-|_{x_1=0},\quad g_N=g_2-g_3\partial_2\hat{\varphi}-g_4\partial_3\hat{\varphi}.
\]
\label{p3.1}
\end{proposition}

\begin{proof}
The proof of equations \eqref{b43} is absolutely the same as the corresponding proof in \cite{T09}. That is, we write down the equation for $\dot{H}^{\pm}$ contained in \eqref{b48}:
\begin{multline}
{\displaystyle
\partial_t\dot{H}^{\pm}+\frac{1}{\partial_1\widehat{\Phi}^{\pm}}\Bigl\{ (\hat{w}^{\pm} ,\nabla )
\dot{H}^{\pm} - (\hat{h}^{\pm} ,\nabla ) \dot{v}^{\pm} + \widehat{H}^{\pm}{\rm div}\,\dot{u}^{\pm}
}\\
+(\dot{u}^{\pm} ,\nabla )\widehat{H}^{\pm} - (\dot{h}^{\pm} ,\nabla ) \hat{v}^{\pm} + \dot{H}^{\pm}{\rm div}\,\hat{u}^{\pm} \Bigr\} ={f}^{\pm}_{H},
\label{aA3}
\end{multline}
where $\dot{u}^{\pm} =(\dot{v}_{N}^{\pm},\dot{v}_2^{\pm}\partial_1\widehat{\Phi}^{\pm}, \dot{v}_3^{\pm}\partial_1\widehat{\Phi}^{\pm})$ and ${f}^{\pm}_{H}=(f_5^{\pm},f_6^{\pm},f_7^{\pm})$. Then, using not only the divergence constraints \eqref{b14'} for the basic state but also equations \eqref{b21} for $\widehat{H}^{\pm}$ themselves, after long calculations, which are omitted, from (\ref{aA3}) we obtain that $f_9^{\pm}={\rm div}\,\dot{h}^{\pm}$ satisfy  (\ref{aa1}) (with $a^{\pm}=f_9^{\pm}/\partial_1\widehat{\Phi}^{\pm}$).

Regarding the additional boundary condition \eqref{b44}, for its derivation we again use (\ref{aA3}) but now being considered at $x_1=0$. Namely, using the boundary conditions (\ref{b50}), system (\ref{b21}) and  constraints (\ref{b14'}) at $x_1=0$ as well as \eqref{b15},  from (\ref{aA3}) being considered at $x_1=0$ we get equation (\ref{aa2}) for $g_8=[\dot{H}_N]$. That is, the proof of Proposition \ref{p3.1} is complete.
\end{proof}

\subsection{Constant coefficients linearized problem for a planar discontinuity}
\label{constcoeff}

If we ``freeze'' coefficients of problem \eqref{b48}--\eqref{b51}, drop the zero-order terms in \eqref{b48} and the zero-order terms in $\varphi$ in \eqref{b50}, assume that $\partial_t\hat{\varphi}=\partial_2\hat{\varphi}=\partial_3\hat{\varphi}=0$, and in the change of variables take $\chi\equiv 1$, then we obtain a constant coefficients linear problem which is the result of the linearization of the original nonlinear free boundary value problem \eqref{3Dsys}, \eqref{bcond}, \eqref{indat} about its {\it exact} piecewise constant solution
\[
U^{\pm}=\widehat{U}^{\pm}=(\hat{p},0,\hat{v}_2,\hat{v}_3,\widehat{H}_1,\widehat{H}_2,\widehat{H}_3,\widehat{S}^{\pm})={\rm const}\qquad \mbox{for}\ x_1\gtrless0
\]
for the planar contact discontinuity $x_1=0$. This exact piecewise constant solution satisfies \eqref{5} and \eqref{cdass}:\footnote{For the constant coefficients problem we do not restrict ourselves to the case of a polytropic gas and consider a general equation of state $\rho =\rho (p,S)$.}
\[
\hat{\rho}^{\pm}=\rho (\hat{p},\widehat{S}^{\pm})>0,\quad (\hat{c}^{\pm})^2= 1/\rho_p (\hat{p},\widehat{S}^{\pm})>0, \quad \widehat{H}_1\neq 0.
\]
Since we can perform the Galilean transformation
\[
\tilde{x}_2=x_2 -\hat{v}_2t,\quad \tilde{x}_3=x_3 -\hat{v}_3t,
\]
without loss of generality we may assume that $\hat{v}_2=\hat{v}_3=0$.

Taking into account the above, we have the following constant coefficients problem:
\begin{equation}
\widehat{A}_0^{\pm}\partial_t{U}^{\pm}\pm \widehat{A}_1\partial_1{U}^{\pm}+\sum_{j=2}^{3}\widehat{A}_j\partial_j{U}^{\pm}=f^{\pm} \qquad \mbox{in}\ \Omega_T,\label{99}
\end{equation}
\begin{equation}
\left(
\begin{array}{c}
{p}^+-{p}^- \\[3pt]
{v}^+-{v}^-\\
{H}_2^+-{H}_2^-\\[3pt]
{H}_3^+-{H}_3^-\\[3pt]
\partial_t\varphi -{v}_1^+
\end{array}
\right)=g \qquad \mbox{on}\ \partial\Omega_T,
\label{100}
\end{equation}
\begin{equation}
({U}^+,{U}^-,\varphi )=0\qquad \mbox{for}\ t<0,\label{101}
\end{equation}
where $\widehat{A}_0^\pm= {\rm diag} (1/(\hat{\rho}^{\pm}(\hat{c}^{\pm})^2) ,\hat{\rho}^{\pm} ,\hat{\rho}^{\pm} ,\hat{\rho}^{\pm} ,1 ,1,1,1)$,
\[
\widehat{A}_1=\left( \begin{array}{cccccccc} 0&1&0&0&0&0&0&0\\[6pt]
1& 0&0&0&0&\widehat{H}_2&\widehat{H}_3&0\\
0&0& 0&0&0&-\widehat{H}_1&0&0\\ 0&0&0& 0&0&0&-\widehat{H}_1&0\\
0&0&0&0&0&0&0&0\\
0&\widehat{H}_2&-\widehat{H}_1&0&0&0&0&0\\
0&\widehat{H}_3&0&-\widehat{H}_1&0&0&0&0\\ 0&0&0&0&0&0&0&0\\
\end{array} \right) ,
\]
and the matrices $\widehat{A}_2$ and $\widehat{A}_3$ can be also easily written down.

We can  reduce \eqref{99}--\eqref{101} to a problem with homogeneous boundary conditions, i.e., with $g=0$. To this end we use the classical argument suggesting to subtract from the solution a more regular function $\widetilde{U}^{\pm}\in H^1(\Omega_T)$ satisfying the boundary conditions. Then, the new unknown $U^{\natural\pm}={U}^\pm-\widetilde{U}^\pm$ satisfies problem \eqref{99}--\eqref{101} with $g=0$ and $f^\pm =F^\pm$, where
\[
F^\pm =f^\pm-\widehat{A}_0^{\pm}\partial_t\widetilde{U}^\pm \mp \widehat{A}_1\partial_1\widetilde{U}^\pm-\sum_{j=2}^{3}\widehat{A}_j\widetilde{U}^\pm
\]
and
\begin{multline}
\sum_{\pm}\|F^\pm\|_{L^2(\Omega_T)}\leq C  \sum_{\pm}\left\{\|f^\pm \|_{L^{2}(\Omega_T)}+ \|\widetilde{U}^\pm \|_{H^{1}(\Omega_T)}\right\}\\
\leq C\biggl\{\sum_{\pm}\|f^\pm \|_{L^{2}(\Omega_T)}+ \|g\|_{H^{1/2}(\partial\Omega_T)} \biggr\}.
\label{nonhom_F}
\end{multline}
Here and later on $C$ is a constant that can change from line to line, and it may depend from another constants.

Since, by virtue of the homogenous boundary conditions \eqref{100} for $U^{\natural\pm}$, we have
\begin{multline*}
[(\widehat{A}_1U^{\natural},U^{\natural})] = (\widehat{A}_1U^{\natural +},U^{\natural +})|_{x_1=0} -(\widehat{A}_1U^{\natural -},U^{\natural -})|_{x_1=0} \\
=2[v_1^{\natural}(p^{\natural} +\widehat{H}_2H_2^{\natural}+\widehat{H}_3H_3^{\natural})] -2\widehat{H}_1[v_2^{\natural}H_2^{\natural}+v_3^{\natural}H_3^{\natural}]=0,
\end{multline*}
standard simple arguments of the energy method applied to systems \eqref{99} give the conserved integral \cite{BThand}
\[
\frac{d}{dt} \biggl(\sum_{\pm}\int_{\mathbb{R}^3_+}(\widehat{A}_0^\pm U^{\natural \pm},U^{\natural \pm}) dx \biggr) =0
\]
if $F^{\pm}=0$. For the general case when $F^{\pm}\neq 0$, we easily get the priori estimate
\[
\sum_{\pm}\|U^{\natural\pm} \|_{L^{2}(\Omega_T)}\leq C \sum_{\pm}\|F^\pm\|_{L^2(\Omega_T)},
\]
for $U^{\natural\pm}$ with no loss of derivatives from $F^{\pm}$. Taking into account \eqref{nonhom_F}, this estimate gives the a priori estimate
\begin{equation}
\sum_{\pm}\|U^{\pm} \|_{L^{2}(\Omega_T)}\leq C \biggl\{\sum_{\pm}\|f^\pm \|_{L^2(\Omega_T)}+ \|g\|_{H^{1/2}(\partial\Omega_T)}\biggr\}
\label{orig_U}
\end{equation}
for the original unknown $U^\pm$ with the loss of ``1/2 derivative'' from $g$ (regarding an a priori estimate for the front perturbation $\varphi$, see Remark \ref{r7} below).

It follows from estimate \eqref{orig_U} that Kelvin-Helmholtz instability never happens for planar contact MHD discontinuities, i.e., they are always (at least, neutrally) stable. It is interesting to note that it seems technically impossible to show this by the spectral analysis because already at the first stage we come to the dispersion relations
\[
\det (s\widehat{A}_0^{\pm}+\lambda \widehat{A}_1 + i\omega_2\widehat{A}_2+i\omega_3\widehat{A}_3)=0
\]
which cannot be analytically solved for the eigenvalues $\lambda$ (even in the 2D planar case \cite{T}). Here $s=\eta +i\xi$, with $\eta >0$, $\xi\in \mathbb{R}$, and $\omega =(\omega_2, \omega_3)\in\mathbb{R}^2$ are the Laplace and Fourier variables respectively.

The fact that the uniform Kreiss-Lopatinski condition \cite{Kreiss} is never satisfied, i.e., planar contact MHD discontinuities are only neutrally stable follows from the non-ellipticity of the front symbol discussed in Section \ref{intro}. Indeed, as was noticed in \cite{IT}, the non-ellipticity of the front symbol always implies the existence of neutral modes. For contact MHD discontinuities, this neutral mode is $s=0$.\footnote{In the original reference frame where $\hat{v}'=(\hat{v}_2,\hat{v}_3)$ is not necessarily zero this neutral mode is $s=-i(\omega,\hat{v}')$.} It corresponds to the following non-trivial normal mode solution of problem \eqref{99}, \eqref{100} with $f^\pm=0$ and $g=0$:
\[
\varphi =\bar{\varphi}e^{i(\omega ,x')},\quad (p^\pm ,v^\pm ,H^\pm )=0,\quad S^{\pm}=\bar{S}^{\pm}e^{i(\omega ,x')},
\]
where $\bar{S}^{\pm}$ and $\bar{\varphi}$ are arbitrary constants. It is natural that in our case when the Kreiss-Lopatinski condition holds only in a weak sense we have a loss of derivatives in the a priori estimate \eqref{orig_U}.

\begin{remark}
{\rm
For estimating the front perturbation $\varphi$ we have to prolong problem \eqref{99}, \eqref{100} up to first-order tangential derivatives (with respect to $t$, $x_2$ and $x_3$). Then, we can estimate the $L^2$ norms of the tangential derivatives of $U^{\pm}$. Taking into account the structure of the boundary matrix $\widehat{A}_1$, we can express $\partial_1v_1^{\pm}$ through these tangential derivatives. After this, using the last boundary condition in \eqref{100} and the trace theorem for $v_1^+$, we get an a priori estimate for $\varphi$ in $L^2$ (see also \cite{BThand}).}
\label{r7}
\end{remark}

\subsection{Linearized problem for the 2D planar case and main result}

In the rest of the paper, we restrict ourselves to 2D planar MHD flows, i.e., when
\[
x=(x_1,x_2)\in \mathbb{R}^2,\quad v=(v_1,v_2)\in \mathbb{R}^2,\quad H=(H_1,H_2)\in \mathbb{R}^2.
 \]
Then, for a polytropic gas the symmetric form of the MHD equations is system \eqref{4}. For the 2D planar case the counterpart of the linearized problem \eqref{b48}--\eqref{b51} reads:
\begin{equation}
A_0(\widehat{U}^{\pm})\partial_t\dot{U}^{\pm} +\widetilde{A}_1(\widehat{U}^{\pm},\widehat{\Psi}^{\pm})\partial_1\dot{U}^{\pm}+A_2(\widehat{U}^{\pm} )\partial_2\dot{U}^{\pm} \\ +\mathcal{C}(\widehat{U}^{\pm},\widehat{\Psi}^{\pm})\dot{U}^{\pm} =f^{\pm}\qquad \mbox{in}\ \Omega_T,\label{b48^}
\end{equation}
\begin{equation}
\left(
\begin{array}{c}
\dot{p}^+-\dot{p}^- + \varphi [\partial_1\hat{p}]\\[3pt]
\dot{v}_1^+-\dot{v}_1^-\\[3pt]
\dot{v}_2^+-\dot{v}_2^-\\[3pt]
\dot{H}_{\tau}^+-\dot{H}_{\tau}^-+ \varphi [\partial_1\widehat{H}_{\tau}]\\[3pt]
\partial_t\varphi +\hat{v}_2^+\partial_2\varphi -\dot{v}_{N}^+ - \varphi \partial_1\hat{v}_N^+
\end{array}
\right)=g \qquad \mbox{on}\ \partial\Omega_T,
\label{b50^}
\end{equation}
\begin{equation}
(\dot{U}^+,\dot{U}^-,\varphi )=0\qquad \mbox{for}\ t<0,\label{b51^}
\end{equation}
where
\[
\Omega_T:= (-\infty, T]\times\mathbb{R}^2_+,\quad \partial\Omega_T:=(-\infty ,T]\times\{x_1=0\}\times\mathbb{R},
\]
\[
\widetilde{A}_1(\widehat{U}^{\pm},\widehat{\Psi}^{\pm})=\frac{1}{\partial_1\widehat{\Phi}^{\pm}}\left(
A_1(\widehat{U}^{\pm})-A_0(\widehat{U}^{\pm})\partial_t\widehat{\Psi}^{\pm}-A_2(\widehat{U}^{\pm})\partial_2\widehat{\Psi}^{\pm}\right),
\]
the matrices $A_0$, $A_1$ and $A_2$ are described just after system \eqref{4},
\[
\dot{U}^{\pm}=(\dot{p}^{\pm},\dot{v}_1^\pm ,\dot{v}_2^\pm ,\dot{H}_1^\pm ,\dot{H}_2^\pm ,\dot{S}^\pm ) ,\quad
f^{\pm} =(f^{\pm}_1,\ldots ,f^{\pm}_6),\quad g=(g_1,\ldots , g_5),
\]
\[
\dot{v}_{N}^{\pm}=
\dot{v}_1^{\pm}-\dot{v}_2^{\pm}\partial_2\widehat{\Psi}^{\pm}, \quad \hat{v}_{N}^{\pm}=\hat{v}_1^{\pm}-\hat{v}_2^{\pm}\partial_2\widehat{\Psi}^{\pm},\quad
\dot{H}_{\tau}^{\pm}=\dot{H}_1^{\pm}\partial_2\widehat{\Psi}^{\pm} +\dot{H}_2^{\pm},\quad \mbox{etc.}
\]

In view of the counterparts of \eqref{boundmatr} and \eqref{boundmatr'} for the 2D planar case, the matrices $\widetilde{A}_1(\widehat{U}^{\pm},\widehat{\Psi}^{\pm})$ have the following structure:
\begin{equation}
\widetilde{A}_1(\widehat{U}^{\pm},\widehat{\Psi}^{\pm})=\widehat{\mathcal{A}}_1^\pm +\widehat{\mathcal{A}}_0^\pm ,
\label{bm_2d}
\end{equation}
where $\widehat{\mathcal{A}}_0^\pm |_{x_1=0}=0$ (the exact form of the matrices $\widehat{\mathcal{A}}_0^\pm$ is of no interest) and
\begin{equation}
(\widehat{\mathcal{A}}_1^\pm \dot{U}^{\pm},\dot{U}^{\pm})=((\widehat{J}^\pm )^T \widehat{\mathcal{A}}_1^\pm\widehat{J}^\pm\dot{W}^\pm,\dot{W}^\pm )=
(\widehat{\mathcal{B}}_1^\pm \dot{W}^\pm,\dot{W}^\pm ),
\label{bm_2d'}
\end{equation}
\begin{equation}
\widehat{\mathcal{B}}_1^{\pm}=\frac{1}{\partial_1\widehat{\Phi}^{\pm}}\begin{pmatrix}
0 & e_1 & 0 & 0 \\[3pt]
e_1^T &O_2& -\widehat{H}_N^{\pm}a_0^\pm & 0^T\\[3pt]
0^T & -\widehat{H}_N^{\pm}a_0^\pm & O_2& 0^T\\[3pt]
0 & 0 &0 & 0
\end{pmatrix},
\label{bm_2d"}
\end{equation}
with $e_1= (1,0)$, $\dot{U}^\pm =\widehat{J}^\pm \dot{W}^\pm$, $\det \widehat{J}^\pm \neq 0$,
\begin{equation}
\dot{W}^{\pm}= (\dot{q}^{\pm},\dot{v}^{\pm}_N,\dot{v}_2^{\pm},\dot{H}^{\pm}_N,\dot{H}^{\pm}_2,\dot{S}^{\pm}),\quad \dot{q}^{\pm}=\dot{p}^{\pm}+\widehat{H}_1^{\pm}\dot{H}_1^{\pm}+\widehat{H}_2^{\pm}\dot{H}_2^{\pm},
\label{bm_2d^}
\end{equation}
and
\[
a_0^{\pm}=\begin{pmatrix}
1 & \partial_2\widehat{\Psi}^{\pm}  \\[3pt]
\partial_2\widehat{\Psi}^{\pm} & 1+(\partial_2\widehat{\Psi}^{\pm})^2
\end{pmatrix} >0.
\]
Moreover, for writing down the quadratic forms with the matrices $\widetilde{A}_1(\widehat{U}^{\pm},\widehat{\Psi}^{\pm})$ on the boundary we will use their exact form (cf. \eqref{A1tilde})
\begin{equation}
\widetilde{A}_1(\widehat{U}^{\pm},\widehat{\Psi}^{\pm})|_{x_1=0} = \pm \begin{pmatrix} 0 & 1 & -\partial_2\hat{\varphi} & 0& 0 & 0\\[6pt]
1 &0 &0 &\widehat{H}_2^\pm\partial_2\hat{\varphi} & \widehat{H}_2^\pm & 0 \\[6pt]
-\partial_2\hat{\varphi} & 0 & 0 & -\widehat{H}_1^\pm\partial_2\hat{\varphi} & -\widehat{H}_1^\pm & 0 \\[6pt]
0 & \widehat{H}_2^\pm\partial_2\hat{\varphi} & -\widehat{H}_1^\pm\partial_2\hat{\varphi} & 0& 0& 0 \\[6pt]
0& \widehat{H}_2^\pm & -\widehat{H}_1^\pm & 0 & 0 & 0 \\[6pt]
0 & 0 &0 &0 &0 &0
\end{pmatrix}_{|x_1=0}.
\label{A1tilde2D}
\end{equation}

As for the case of constant coefficients in \eqref{99}--\eqref{101}, for problem \eqref{b48^}--\eqref{b51^} we will not use the boundary constraint associated with \eqref{15}. Therefore, we do not include it in the 2D counterpart of Proposition \ref{p3.1} below.

\begin{proposition}
Let the basic state \eqref{a21} satisfies assumptions \eqref{a5}--\eqref{avn} (to be exact, their 2D planar analogs).
Then solutions of problem \eqref{b48^}--\eqref{b51^} satisfy
\begin{equation}
{\rm div}\,\dot{h}^+=f_7^+,\quad {\rm div}\,\dot{h}^-=f_7^-\quad\mbox{in}\ \Omega_T.
\label{b43^}
\end{equation}
Here $\dot{h}^{\pm}=(\dot{H}_{N}^{\pm},\dot{H}_2^{\pm}\partial_1\widehat{\Phi}^{\pm})$ and the functions $f_7^{\pm}=f_7^{\pm}(t,x )$, which vanish in the past, are determined by the source terms and the basic state as solutions to the linear inhomogeneous equations
\begin{equation}
\partial_t a^{\pm}+ \frac{1}{\partial_1\widehat{\Phi}^{\pm}}\left\{ (\hat{w}^{\pm} ,\nabla a^{\pm}) + a^{\pm}\,{\rm div}\,\hat{u}^{\pm}\right\}={\mathcal F}^{\pm}\quad \mbox{in}\ \Omega_T,
\label{aa1^}
\end{equation}
where $a^{\pm}=f_7^{\pm}/\partial_1\widehat{\Phi}^{\pm},\quad {\mathcal F}^{\pm}=({\rm div}\,{f}_{h}^{\pm})/\partial_1\widehat{\Phi}^{\pm}$,
${f}_{h}^{\pm}=(f_{N}^{\pm} ,\partial_1\widehat{\Phi}^{\pm}f_5^{\pm})$, $f_{N}^{\pm}=f_4^{\pm}-f_5^{\pm}\partial_2\widehat{\Psi}^{\pm}$, and the vectors $\hat{w}^\pm$ and $\hat{u}^\pm$ are the 2D analogs of the corresponding vectors introduced above.
\label{p3.1'}
\end{proposition}

\begin{remark}
{\rm
In view of \eqref{a12'}, $\hat{w}_1^\pm |_{x_1=0}=0$. Then equations \eqref{aa1^} do not need any boundary conditions and from them we easily get the estimates
\begin{equation}
\|f_7^\pm (t)\|_{L^2(\mathbb{R}^2_+)}\leq C\|{\rm div}\,{f}_{h}^{\pm}\|_{L^2(\Omega_T)}\leq C\|f^\pm \|_{H^1(\Omega_T)}
\label{f7'}
\end{equation}
and
\begin{equation}
\|f_7^\pm \|_{L^2(\Omega_T)}\leq  C\|f^\pm \|_{H^1(\Omega_T)}.
\label{f7}
\end{equation}
}
\label{rf7}
\end{remark}

We are now in a position to state the main result of this paper.

\begin{theorem}
Let the basic state \eqref{a21} satisfies assumptions \eqref{a5}--\eqref{avn} (in the 2D planar case). Let also
\begin{equation}
[\partial_1\hat{p} ]\geq \epsilon >0
\label{RTL}
\end{equation}
where $[\partial_1\hat{p} ]=\partial_1\hat{p}^+_{|x_1=0} +\partial_1\hat{p}^-_{|x_1=0}$ (see \eqref{norm_jump}) and $\epsilon$ is a fixed constant. Then, for all $f^{\pm} \in H^1(\Omega_T)$ and $g\in H^{3/2}(\partial\Omega_T)$ which vanish in the past, problem \eqref{b48^}--\eqref{b51^} has a unique solution $((\dot{U}^+, \dot{U}^-) ,\varphi )\in H^1(\Omega_T)\times H^1(\partial\Omega_T)$. Moreover, this solution obeys the a priori estimate
\begin{equation}
\sum_{\pm}\|\dot{U}^\pm \|_{H^{1}(\Omega_T)}+\|\varphi\|_{H^1(\partial\Omega_T)} \leq C\biggl\{\sum_{\pm}\|f^\pm \|_{H^{1}(\Omega_T)}+ \|g\|_{H^{3/2}(\partial\Omega_T)}\biggr\},
\label{main_est}
\end{equation}
where $C=C(K,\bar{\rho}_0,\bar{p}_0,\bar{\kappa}_0,\epsilon,T)>0$ is a constant independent of the data $f^\pm$ and $g$.
\label{t1}
\end{theorem}

Note that inequality \eqref{RTL} is just the Rayleigh-Taylor sign condition \eqref{RT} written for the straightened unperturbed discontinuity (with the equation $x_1=0$).

\section{Energy a priori estimate for the 2D planar case}
\label{sec:4}

\subsection{Reduction to homogeneous boundary conditions}

As for the case of constant coefficients in subsection \ref{constcoeff}, we reduce \eqref{b48^}--\eqref{b51^} to a problem with homogeneous boundary conditions, i.e., with and $g=0$. Using again the classical argument, we subtract from the solution a more regular function $\widetilde{U}^{\pm}\in H^2(\Omega_T)$ satisfying the boundary conditions \eqref{b50^}. Then, the new unknown
\begin{equation}
U^{\pm\natural}=\dot{U}^\pm-\widetilde{U}^\pm ,\label{a87'}
\end{equation}
with
\begin{equation}
\|\widetilde{U}^\pm \|_{H^1(\Omega_T)}\leq C\|g \|_{H^{1/2}(\partial\Omega_T)},
\label{tildU}
\end{equation}
satisfies problem \eqref{b48^}--\eqref{b51^} with $f^\pm =F^\pm =(F^\pm_1,\ldots , F^\pm_6)$, where
\begin{equation}
F^\pm =f^\pm-A_0(\widehat{U}^{\pm})\partial_t\widetilde{U}^{\pm} -\widetilde{A}_1(\widehat{U}^{\pm},\widehat{\Psi}^{\pm})\partial_1\widetilde{U}^{\pm}  -A_2(\widehat{U}^{\pm} )\partial_2\widetilde{U}^{\pm}-\mathcal{C}(\widehat{U}^{\pm},\widehat{\Psi}^{\pm})\widetilde{U}^\pm
\label{a87''}
\end{equation}
and $F^\pm$ satisfy estimates (cf. \eqref{nonhom_F}):
\begin{equation}
\sum_{\pm}\|F^\pm\|_{H^1(\Omega_T)}\leq C\biggl\{\sum_{\pm}\|f^\pm \|_{H^1(\Omega_T)}+ \|g\|_{H^{3/2}(\partial\Omega_T)} \biggr\}.
\label{nonhom_F'}
\end{equation}

Dropping for convenience the indices $^{\natural}$ in \eqref{a87'}, we get our reduced linearized problem:
\begin{equation}
A_0(\widehat{U}^{\pm})\partial_t{U}^{\pm} +\widetilde{A}_1(\widehat{U}^{\pm},\widehat{\Psi}^{\pm})\partial_1{U}^{\pm}+A_2(\widehat{U}^{\pm} )\partial_2{U}^{\pm} +\mathcal{C}(\widehat{U}^{\pm},\widehat{\Psi}^{\pm}){U}^{\pm} =F^{\pm}\qquad \mbox{in}\ \Omega_T,\label{b48b}
\end{equation}
\begin{eqnarray}
 {[}p{]}=- \varphi {[}\partial_1\hat{p}{]}, & \label{b50b.1}\\[3pt]
  {[}v{]}=0, & \label{b50b.2}\\[3pt]
  {[}H_{\tau}{]}=-\varphi {[}\partial_1\widehat{H}_{\tau}{]}, & \label{b50b.3} \\[3pt]
{v}_{N}^+ = \partial_0\varphi - \varphi \partial_1\hat{v}_N^+  & \qquad \mbox{on}\ \partial\Omega_T
\label{b50b.4}
\end{eqnarray}
\begin{equation}
({U}^+,{U}^-,\varphi )=0\qquad \mbox{for}\ t<0,\label{b51b}
\end{equation}
where
\begin{equation}
\partial_0:=  \partial_t +\hat{v}_2^+\partial_2 \qquad \mbox{in}\ \Omega_T.
\label{d0}
\end{equation}
It is not a big mistake to call $\partial_0$ the material derivative because on the boundary $\partial_0$ coincides with the material derivative $\partial_t + (\hat{w}^\pm ,\nabla )$ (in the reference frame related to the discontinuity). In view of \eqref{b43^} and \eqref{f7}, solutions to problem \eqref{b48b}--\eqref{b51b} satisfy
\begin{equation}
{\rm div}\,{h}^\pm=F_7^\pm\quad\mbox{in}\ \Omega_T
\label{b43b}
\end{equation}
where
\begin{equation}
\|F_7^\pm \|_{L^2(\Omega_T)}\leq  C\|F^\pm \|_{H^1(\Omega_T)}
\label{f7^}
\end{equation}
and we will also need the layerwise counterpart of \eqref{f7^} (cf. \eqref{f7'})
\begin{equation}
\|F_7^\pm (t)\|_{L^2(\mathbb{R}^2_+)}\leq  C\|F^\pm \|_{H^1(\Omega_T)}.
\label{f7"}
\end{equation}

Taking into account \eqref{tildU} and \eqref{nonhom_F'}, the following lemma for the reduced problem \eqref{b48b}--\eqref{b51b} yields estimate \eqref{main_est} for problem \eqref{b48^}--\eqref{b51^}.

\begin{lemma}
Let the basic state \eqref{a21} satisfies assumptions \eqref{a5}--\eqref{avn} (in the 2D planar case) together with condition \eqref{RTL}.
Then for all $F^{\pm} \in H^1(\Omega_T)$  which vanish in the past the a priori estimate
\begin{equation}
\sum_{\pm}\|{U}^\pm \|_{H^{1}(\Omega_T)}+\|\varphi\|_{H^1(\partial\Omega_T)} \leq C\sum_{\pm}\|F^\pm \|_{H^{1}(\Omega_T)}
\label{main_est'}
\end{equation}
holds for problem \eqref{b48b}--\eqref{b51b}, where $C=C(K,\bar{\rho}_0,\bar{p}_0,\bar{\kappa}_0,\epsilon,T)>0$ is a constant independent of the data $F^\pm$.
\label{lem1}
\end{lemma}

That is, it remains to prove Lemma \ref{lem1} and the rest of this section will be devoted to this proof. In what follows we assume by default that the assumptions of Lemma \ref{lem1} are satisfied.

\subsection{Estimate of normal derivatives through tangential ones} In spite of the fact that for problem \eqref{b48b}--\eqref{b51b} the boundary $x_1=0$ is characteristic, using the structure of the boundary matrix (see \eqref{bm_2d}--\eqref{bm_2d^}) and the divergence constraints \eqref{b43b}, we can estimate the $L^2-$norms of the normal derivatives $\partial_1U^{\pm}$ through the $L^2-$norms of $U^\pm$ and the tangential derivatives $\partial_tU^{\pm}$ and $\partial_2U^\pm $ (and the $H^1-$ norms of the source terms $F^\pm$).

\begin{proposition}
The solutions to problem \eqref{b48b}--\eqref{b51b} obey the estimate
\begin{equation}
\|\partial_1{U}^\pm\|^2_{L^2(\Omega_t)}  \leq C\left\{ \|F^\pm \|^2_{H^1(\Omega_T)} +\int_0^tI^\pm (s)ds \right\}
\label{d1U}
\end{equation}
for all $t\leq T$ and $C=C(K,\bar{\rho}_0,\bar{p}_0,\bar{\kappa}_0,T)>0$ being a constant independent of the source terms $F^\pm$, where $\Omega_t= (-\infty ,t]\times\mathbb{R}^2_+$ and
\begin{equation}
I^\pm (t) =\|U^\pm (t)\|^2_{L^2(\mathbb{R}^2_+)}+
\|\partial_t U^\pm (t)\|^2_{L^2(\mathbb{R}^2_+)}+ \|\partial_2U^\pm (t)\|^2_{L^2(\mathbb{R}^2_+)}
\label{Ipm}
\end{equation}
(clearly, $
\int_0^tI^\pm (s)ds =
\|U^\pm\|^2_{L^2(\Omega_t)}+
\|\partial_t U^\pm\|^2_{L^2(\Omega_t)}+ \|\partial_2U^\pm\|^2_{L^2(\Omega_t)}$).
\label{p5}
\end{proposition}

\begin{proof}
First of all, for any linear symmetric hyperbolic system in the half-plane $\mathbb{R}^2_+$ we can easily get an estimate for the weighted normal derivative $\sigma\partial_1$ of the unknown,  where the weight $\sigma (x_1)\in C^{\infty}(\mathbb{R}_+)$ is a monotone increasing function such that $\sigma (x_1)=x_1$ in a neighborhood of the origin and $\sigma (x_1)=1$ for $x_1$ large enough. Indeed, to estimate the weighted normal derivatives of $U^\pm$ we do not need boundary conditions because the weight $\sigma |_{x_1=0}=0$. By applying to systems \eqref{b48b} the operator $\sigma\partial_1$ and using standard arguments of the energy method, we obtain the inequality
\begin{equation}
\|\sigma\partial_1U^{\pm} (t)\|^2_{L^2(\mathbb{R}^2_+)}\leq C\biggl\{ \|F^\pm \|^2_{H^1(\Omega_T)} + \int_0^{\tau}I^\pm (s)ds  +\int_0^{t}\|\sigma\partial_1U^{\pm} (s)\|^2_{L^2(\mathbb{R}^2_+)}ds\biggr\}
\label{4.92}
\end{equation}
for all $t\leq \tau \leq T$. Applying then Gronwall's lemma to the function of $t$ staying in the left-hand side of \eqref{4.92}, we get
\[
\|\sigma\partial_1U^{\pm} (t)\|^2_{L^2(\mathbb{R}^2_+)}\leq C\biggl\{ \|F^\pm \|^2_{H^1(\Omega_T)} + \int_0^{\tau}I^\pm (s)ds\biggr\}.
\]
Integrating the last inequality over the time interval $(0,t )$ and setting $\tau =t$ in the result, we come to the estimate
\begin{equation}
\|\sigma\partial_1{U}^\pm\|^2_{L^2(\Omega_t)}  \leq C\left\{ \|F^\pm \|^2_{H^1(\Omega_T)} +\int_0^tI^\pm (s)ds \right\}
\label{sigma_d1U}
\end{equation}
for all $t\leq T$ and with a constant $C=C(K,\bar{\rho}_0,\bar{p}_0,T)>0$.

Taking into account the structure of the boundary matrices $\widetilde{A}_1(\widehat{U}^\pm , \widehat{\Psi}^\pm )$ detailed in \eqref{bm_2d}--\eqref{bm_2d^}, we can estimate $\partial_1U^{\pm}$ in a neighborhood of the boundary. Indeed, in view of assumption \eqref{cdass} and the continuity of the basic state \eqref{a21}, there exist such a small but fixed constant $\delta >0$ depending on $\bar{\kappa}_0$ that
\begin{equation}
|\widehat{H}_N^{\pm}|\geq \frac{\bar{\kappa}_0}{2}>0 \qquad\mbox{in}\ \Omega_t^{\delta},
\label{cdass"}
\end{equation}
where
\[
\Omega_t^{\delta}=(-\infty ,t]\times \mathbb{R}^2_{\delta},\quad \mathbb{R}^2_{\delta} =(0,\delta )\times\mathbb{R}.
\]
Rewriting systems \eqref{b48b} in terms of the vectors
\[
{W}^{\pm}= ({q}^{\pm},{v}^{\pm}_N,{v}_2^{\pm},{H}^{\pm}_N,{H}^{\pm}_2,{S}^{\pm})
\]
(with ${q}^{\pm}={p}^{\pm}+\widehat{H}_1^{\pm}{H}_1^{\pm}+\widehat{H}_2^{\pm}{H}_2^{\pm}$, see \eqref{bm_2d^}), using the divergence constraints \eqref{b43b} and taking into account the structure of the matrices $\widehat{\mathcal{B}}_1^\pm$ in \eqref{bm_2d"}, by virtue of \eqref{cdass"}, we can resolve the rewritten systems for the normal derivatives of $(q^\pm ,v_N^\pm ,v_2^\pm ,H_N^\pm,H_2^\pm )$  in $\Omega_t^{\delta}$. By returning to the original unknowns and applying \eqref{f7"} this gives the estimates
\begin{multline}
\| \partial_1V^\pm (t)\|^2_{L^2(\mathbb{R}^2_{\delta})}
\leq C \left\{ \|F^\pm (t)\|^2_{L^2(\mathbb{R}^2_+)} + \|F_7^\pm (t)\|^2_{L^2(\mathbb{R}^2_+)}+ I^\pm (t)\right\}
\\
\leq C \left\{ \|F^\pm \|^2_{H^1(\Omega_T)} + I^\pm (t)\right\}
\label{1Vt}
\end{multline}
and
\begin{equation}
\| \partial_1V^\pm \|^2_{L^2(\Omega_t^{\delta})} \leq C \left\{ \|F^\pm \|^2_{H^1(\Omega_T)} + \int_0^t I^\pm (s)ds\right\}
\label{1V}
\end{equation}
with a constant $C=C(K,\bar{\kappa}_0,T)>0$, where
\[
V^\pm = (p^\pm ,v^\pm ,H^\pm).
\]
While deriving \eqref{1Vt} we exploited the elementary inequality
\begin{equation}
\| F^\pm (t)\|^2_{L^2(\mathbb{R}^2_+)} \leq \| F^\pm \|^2_{L^2(\Omega_t)} +\| \partial_tF^\pm \|^2_{L^2(\Omega_t)}
\label{elem}
\end{equation}
following from the trivial relation
\[
\frac{d}{dt}\,\|F^\pm (t)\|^2_{L_2(\mathbb{R}^2_+)}=2\int_{\mathbb{R}^2_+}(F^\pm,\partial_tF^\pm)dx
\]
(clearly, \eqref{elem} is roughened as $\| F^\pm (t)\|^2_{L^2(\mathbb{R}^2_+)} \leq \| F^\pm \|^2_{H^1(\Omega_T)}$).

Since the weight $\sigma$ in \eqref{sigma_d1U} is not zero outside the boundary, it follows from estimate \eqref{sigma_d1U} that
\begin{equation}
\|\partial_1{U}^\pm\|^2_{L^2(\Omega_t \setminus \Omega_t^{\delta})}  \leq C\left\{ \|F^\pm \|^2_{H^1(\Omega_T)} +\int_0^tI^\pm (s)ds \right\},
\label{d1U_out}
\end{equation}
where the constant $C$ depends, in particular, on $\delta$ and so on $\bar{\kappa}_0$. Combining \eqref{1V} and \eqref{d1U_out}, we get
\begin{equation}
\|\partial_1{V}^\pm\|^2_{L^2(\Omega_t )}  \leq C\left\{ \|F^\pm \|^2_{H^1(\Omega_T)} +\int_0^tI^\pm (s)ds \right\}.
\label{d1V}
\end{equation}

For obtaining the desired estimate \eqref{d1U} it remains to include the normal derivative of $S^\pm$ into estimate \eqref{d1V}. To this end we consider the last equations in systems \eqref{b48b}:
\begin{equation}
\partial_tS^{\pm}+\frac{1}{\partial_1\widehat{\Phi}^{\pm}} (\hat{w}^{\pm} ,\nabla S^{\pm} ) + \mbox{l.o.t.} =F_6^\pm \quad\mbox{in}\ \Omega_T,
\label{entropy}
\end{equation}
where l.o.t are unimportant lower-order terms. Since $\hat{w}_1^\pm |_{x_1=0}=0$, the linear equations \eqref{entropy} considered as equations for $S^\pm$ do not need any boundary conditions. Writing  equations similar to \eqref{entropy} for $\partial_1 S^{\pm}$ and using \eqref{d1V}  to estimate the normal derivatives of $V^{\pm}$ that appear in the right-hand side of the equations for $\partial_1 S^{\pm}$,  we easily derive the estimate
\begin{equation}
\|\partial_1S^\pm\|^2_{L^2(\Omega_t )}  \leq C\left\{ \|F^\pm \|^2_{H^1(\Omega_T)} +\int_0^tI^\pm (s)ds \right\}.
\label{ent_est}
\end{equation}
Estimates \eqref{d1V} and \eqref{ent_est} give \eqref{d1U}.
\end{proof}

Note that below we will need the ``local'' layerwise estimate \eqref{1Vt} together with \eqref{d1U} for obtaining the a priori estimate \eqref{main_est'}.

\subsection{Preparatory estimation of material derivatives}
We apply the differential operator $\partial_0$ (see \eqref{d0}) to systems \eqref{b48b}. Then, by standard arguments of the energy method applied to the resulting symmetric hyperbolic systems (considered as systems for $\partial_0U^\pm $ with lower-order terms depending on $U^\pm$, $\partial_tU^\pm$, $\partial_1U^\pm$ and $\partial_2U^\pm$), in view of \eqref{d1U},  we obtain
\begin{multline}
\sum_{\pm}\int_{\mathbb{R}^2_+}\bigl(A_0(\widehat{U}^\pm )\partial_0U^\pm , \partial_0U^\pm \bigr)(t)\,dx +2 \int_{\partial\Omega_t}Q_0dx_2ds \\
\leq C\biggl\{ \sum_{\pm}\|F^\pm \|^2_{H^1(\Omega_T)} +\int_0^tI (s)ds \biggr\},
\label{q0}
\end{multline}
where $I(t)=I^+(t)+I^-(t)$, with $I^\pm (t)$ defined in \eqref{Ipm}, and
\[
Q_0 =-\frac{1}{2}\sum_{\pm}\bigl(\widetilde{A}_1(\widehat{U}^\pm, \widehat\Psi^{\pm} )\partial_0U^\pm , \partial_0U^\pm \bigr)\bigr|_{x_1=0}\,.
\]

Using \eqref{A1tilde2D} and the boundary conditions \eqref{b50b.2}, we first calculate
\begin{multline}
-\frac{1}{2}\sum_{\pm}\bigl(\widetilde{A}_1(\widehat{U}^\pm, \widehat\Psi^{\pm} )U^\pm , U^\pm \bigr)\bigr|_{x_1=0}\\=\left.\left(-v_N^+[p]+
(\widehat{H}_1^+v_2^+-\widehat{H}_2^+v_1^+)[H_{\tau}]\right)\right|_{x_1=0}\\
=\left.\left(-v_N^+[p]+
(\widehat{H}_N^+v_2^+-\widehat{H}_2^+v_N^+)[H_{\tau}]\right)\right|_{x_1=0}.
\label{quadr}
\end{multline}
Then
\begin{multline*}
Q_0=
\left.\left(-\partial_0v_N^+[\partial_0p]+
(\widehat{H}_N^+\partial_0v_2^+-\widehat{H}_2^+\partial_0v_N^+)[\partial_0H_{\tau}]\right)\right|_{x_1=0}  +
\big({\rm coeff}\,v_2^+[\partial_0p]  \\ +{\rm coeff}\,v_2^+ [\partial_0H_{\tau}] + {\rm coeff}\,[H_1] \partial_0v_2^+ + {\rm coeff}\,[H_1] \partial_0v_N^+ + {\rm coeff}\,v_2^+[H_1]\big)\big|_{x_1=0}.
\end{multline*}
Here and later on \,coeff\, denotes a coefficient that can change from line to line, is determined by the basic state, and its concrete form is of no interest.
Taking into account the boundary conditions \eqref{b50b.1}, \eqref{b50b.3}, \eqref{b50b.4}, we obtain
\begin{equation}
\partial_0[p] = {\rm coeff}\,\partial_0\varphi + {\rm coeff}\,\varphi \\= {\rm coeff}\,v_N^+|_{x_1=0} + {\rm coeff}\,\varphi ,
\label{d0q}
\end{equation}
\begin{equation}
\partial_0[H_{\tau}] = {\rm coeff}\,v_N^+|_{x_1=0} + {\rm coeff}\,\varphi .
\label{d0H2}
\end{equation}
By substituting \eqref{d0q} and \eqref{d0H2} into the quadratic form $Q_0$, one gets
\begin{multline}
Q_0=\bigl({\rm coeff}\,v_N^+\partial_0v_N^+ +{\rm coeff}\,v_N^+\partial_0v_2^+ +{\rm coeff}\,\varphi \partial_0v_N^+ +{\rm coeff}\,\varphi \partial_0v_2^+ \\
+{\rm coeff}\,v_2^+v_N^+ + {\rm coeff}\,v_2^+\varphi + {\rm coeff}\,[H_1] \partial_0v_2^+ \\+ {\rm coeff}\,[H_1] \partial_0v_N^+ +{\rm coeff}\,v_2^+[H_1]
\bigr)\bigr|_{x_1=0}\,.
\label{Q_0}
\end{multline}

To treat the integrals of the ``lower-order'' terms like $\,{\rm coeff} v_N^+\partial_{\alpha}v_2^+|_{x_1=0}$ contained in the right-hand side of \eqref{Q_0}, where  $\alpha =2$ or we have the time derivative $\partial_t$, we use the same standard arguments as in \cite{T05,T09,Tcpam}. That is, we pass to the volume integral and integrate by parts:
\begin{multline}
\int_{\partial\Omega_t}\hat{c}\,{v}_N^+\partial_{\alpha}{v}_2^+|_{x_1=0}\,dx_2ds
=-\int_{\Omega_t}\partial_1\bigl(\tilde{c}{v}_N^+\partial_{\alpha}{v}_2^+\bigr)dxds \\
=\int_{\Omega_t}\Bigl\{\tilde{c}\partial_{\alpha}{v}_N^+\partial_1{v}_2^+ +(\partial_{\alpha}\tilde{c}){v}_N^+\partial_1{v}_2^+ -\tilde{c}\partial_1{v}_N^+\partial_{\alpha}{v}_2^+\\-(\partial_1\tilde{c}){v}_N^+\partial_\alpha{v}_2^+\Bigr\}dxds
-\int_{\Omega_t}\partial_{\alpha}\bigl(\tilde{c}{v}_N^+\partial_1{v}_2^+\bigr)dxds,
\label{intbypart}
\end{multline}
where $\hat{c}$ is a coefficient and $\tilde{c}|_{x_1=0}=\hat{c}$.

If $\alpha =2$ the last integral in \eqref{intbypart} is equal to zero. But if $\partial_{\alpha}$ denotes the time derivative, the last integral does not disappear. In this case we use a cut-off in the passage to the volume integral. That is, we may assume that the coefficient $\tilde{c}$ appearing in the volume integrals in \eqref{intbypart} is zero for $x_1>\delta $, where $\delta$ is the same as in \eqref{cdass"}. For example, if $\hat{c}=\widehat{H}_N^+|_{x_1=0}$, then we take $\tilde{c} = \eta\widehat{H}_N^+$, where  $\eta (x_1)\in C^{\infty}(\mathbb{R}_+)$ is such a rapidly decreasing function that $\eta (0)=1$ and $\eta (x_1)=0$ for $x_1>\delta$.\footnote{If $\hat{c}=1$, we just take $\tilde{c} = \eta$.} Then, all the integrals over $\Omega_t$ in \eqref{intbypart} are replaced by the same integrals over $\Omega^{\delta}_t$ and the last integral reads:
\begin{equation}
-\int_{\Omega_t^{\delta}}\partial_s\bigl(\tilde{c}{v}_N^+\partial_1{v}_2^+\bigr)dxds=
-\int_{\mathbb{R}^2_{\delta}}\left(\tilde{c}{v}_N^+\partial_1{v}_2^+\right)(t)\,dx.
\label{cut-off}
\end{equation}
Using the Young inequality, the elementary inequality (see \eqref{elem})
\begin{equation}
\| U^\pm (t)\|^2_{L^2(\mathbb{R}^2_+)} \leq \int_0^tI (s)ds
\label{elem'}
\end{equation}
and \eqref{1Vt}, we estimate the last integral as follows:
\begin{equation}
-\int_{\mathbb{R}^2_{\delta}}\tilde{c}{v}_N^+\partial_{1}{v}_2^+\,dx\leq C\biggl\{
\sum_\pm\|F^\pm\|^2_{H^1(\Omega_T)}+\tilde{\varepsilon}\,I(t) +\frac{1}{\tilde{\varepsilon}}\,
\int_0^tI (s)ds \biggr\},
\label{intbypart1}
\end{equation}
where $\tilde{\varepsilon} $ is a small positive constant. The last but one integral in \eqref{intbypart} is estimated by using \eqref{d1U}:
\begin{multline}
\int_{\Omega_t}\Bigl\{\tilde{c}\partial_{\alpha}{v}_N^+\partial_1{v}_2^+ +(\partial_{\alpha}\tilde{c}){v}_N^+\partial_1{v}_2^+ -\tilde{c}\partial_1{v}_N^+\partial_{\alpha}{v}_2^+ -(\partial_1\tilde{c}){v}_N^+\partial_\alpha{v}_2^+\Bigr\}dxds \\
\leq C\biggl\{ \sum_\pm \|F^\pm \|^2_{H^1(\Omega_T)} +\int_0^tI (s)ds \biggr\}.
\label{intbypart2}
\end{multline}

The integrals of the terms like  $\,{\rm coeff} \varphi\partial_0 v_N^+|_{x_1=0}$ contained in \eqref{Q_0} are treated by the integration by parts, using the Young inequality, the cut-off argument as above (when we pass to the volume integral), estimates \eqref{d1U} and \eqref{1Vt} and the estimate
\begin{equation}
\|\varphi (t)\|^2_{L^2(\mathbb{R})}\leq C\left\{ \sum_\pm \|F^\pm \|^2_{H^1(\Omega_T)} +\int_0^tI (s)ds+ \|\varphi \|^2_{L^2(\partial\Omega_t)} \right\}
\label{frontL2}
\end{equation}
following from the boundary condition \eqref{b50b.4} and estimate \eqref{d1U}:
\begin{multline}
\int_{\partial\Omega_t}{\rm coeff}\, \varphi\partial_0v_N^+|_{x_1=0}\, dx_2ds=\int_{\mathbb{R}}{\rm coeff}\, \varphi v_N^+|_{x_1=0}\, dx_2 \\ +
\int_{\partial\Omega_t}\left.\left({\rm coeff}\, \varphi v_N^+ + {\rm coeff}\, v_N^+\partial_0\varphi \right)\right|_{x_1=0} dx_2ds   =
\int_{\mathbb{R}}{\rm coeff}\, \varphi v_N^+|_{x_1=0}\, dx_2 \\  +
\int_{\partial\Omega_t}\left.\left({\rm coeff}\, \varphi v_N^+ + {\rm coeff}\, (v_N^+)^2 \right)\right|_{x_1=0} dx_2ds \\ \leq C\biggl\{ \frac{1}{\tilde{\varepsilon}}\,\|\varphi (t)\|^2_{L^2(\mathbb{R})}   + \tilde{\varepsilon}\,\| v_N^+(t) \|^2_{L^2(\mathbb{R}^2_+)} + \tilde{\varepsilon}\,\| \partial_1v_N^+(t) \|^2_{L^2(\mathbb{R}^2_{\delta})}\\ +\|\varphi \|^2_{L^2(\partial\Omega_t)}  +\| v_N^+ \|^2_{L^2(\Omega_t)} + \| \partial_1v_N^+ \|^2_{L^2(\Omega_t)}\biggr\} \\ \leq  \tilde{\varepsilon}C I(t)+\widetilde{C}(\tilde{\varepsilon})\biggl\{\sum_\pm \|F^\pm \|^2_{H^1(\Omega_T)} +\int_0^tI (s)ds+ \|\varphi \|^2_{L^2(\partial\Omega_t)} \biggr\}.
\label{intbypart3}
\end{multline}
Here and below $\widetilde{C}=\widetilde{C}(\tilde{\varepsilon})$ is a positive constant depending on $\tilde{\varepsilon} $.
In \eqref{intbypart3} we also used the boundary condition \eqref{b50b.4} for expressing $\partial_0\varphi$ through $v_N^+|_{x_1=0}$ and $\varphi$.

Thus, taking into account the above, we estimate the quadratic form $Q_0$ as follows:
\begin{equation}
Q_0\leq   \tilde{\varepsilon}C I(t)+\widetilde{C}(\tilde{\varepsilon})\biggl\{\sum_\pm \|F^\pm \|^2_{H^1(\Omega_T)}  +\int_0^tI (s)ds+ \|\varphi \|^2_{L^2(\partial\Omega_t)} \biggr\}.
\label{estQ0}
\end{equation}
Then, taking into account the positive definiteness of the symmetric matrices $A_0(\widehat{U}^\pm )$, from \eqref{q0}, \eqref{elem'}, \eqref{frontL2} and \eqref{estQ0} we obtain
\begin{multline}
\sum_\pm \Bigl\{ \| U^\pm (t)\|^2_{L^2(\mathbb{R}^2_+)} +\| \partial_0U^\pm (t)\|^2_{L^2(\mathbb{R}^2_+)} \Bigr\} + \|\varphi (t)\|^2_{L^2(\mathbb{R})}\\
\leq   \tilde{\varepsilon}C\sum_\pm\Bigl(\| \partial_tU^\pm (t)\|^2_{L^2(\mathbb{R}^2_+)} +\|\partial_2 U^\pm (t)\|^2_{L^2(\mathbb{R}^2_+)}\Bigr)\\
+\widetilde{C}(\tilde{\varepsilon})\left\{\sum_\pm \|F^\pm \|^2_{H^1(\Omega_T)}   +\int_0^tI (s)ds+ \|\varphi \|^2_{L^2(\partial\Omega_t)} \right\}.
\label{estd0}
\end{multline}

\subsection{Estimation of $x_2$-derivatives: deduction of the main energy inequality}
We now differentiate systems \eqref{b48b} with respect to $x_2$. Then, similarly to \eqref{q0} by the energy method  we get
\begin{multline}
\sum_{\pm}\int_{\mathbb{R}^2_+}\bigl(A_0(\widehat{U}^\pm )\partial_2U^\pm , \partial_2U^\pm \bigr)dx +2 \int_{\partial\Omega_t}Q_2dx_2ds \\
\leq C\biggl\{ \sum_{\pm}\|F^\pm \|^2_{H^1(\Omega_T)} +\int_0^tI (s)ds \biggr\},
\label{q2}
\end{multline}
where, cf. \eqref{quadr},
\[
Q_2 =-\frac{1}{2}\sum_{\pm}\bigl(\widetilde{A}_1(\widehat{U}^\pm, \widehat\Psi^{\pm} )\partial_2U^\pm , \partial_2U^\pm \bigr)\bigr|_{x_1=0} \\
=\left.\left(-\partial_2v_N^+[\partial_2p]+
R[\partial_2H_{\tau}]\right)\right|_{x_1=0} +\mathcal{Q}_2
\]
and
\[
R=
(\widehat{H}_N^+\partial_2v_2^+-\widehat{H}_2^+\partial_2v_N^+)|_{x_1=0},
\]
\[
\mathcal{Q}_2 = \big({\rm coeff}\,v_2^+[\partial_2p] +{\rm coeff}\,v_2^+ [\partial_2H_{\tau}] + {\rm coeff}\,[H_1] \partial_2v_2^+  + {\rm coeff}\,[H_1] \partial_2v_N^+ + {\rm coeff}\,v_2^+ [H_1]\big)\big|_{x_1=0}.
\]

To treat the quadratic form $Q_2$ we use not only the boundary conditions but also the interior equations considered on the boundary. Namely, by multiplying the equations for $H^+$ contained in \eqref{b48b} (see their 3D counterpart for $\dot{H} ^\pm$ in \eqref{aA3}) by the vector $(1,-\partial_2\widehat{\Psi}^+)$, considering the result at $x_1=0$ and taking \eqref{b50b.4} into account, we obtain
\begin{equation}
R = -\partial_0H_N^+  +{\rm coeff}\,v_N^+ +{\rm coeff}\,v_2^+ +{\rm coeff}\,H_N^+ +F_N^+\qquad \mbox{on}\ \partial\Omega_T,
\label{eqH}
\end{equation}
with $F_N^+=F_4^+-F_5^+\partial_2\widehat{\Psi}^+$.
Using \eqref{eqH} (as well as the definition of $R$ itself) and the boundary conditions \eqref{b50b.3} and \eqref{b50b.4}, we calculate:
\begin{multline}
Q_2 = [\partial_1\hat{p}]\partial_2\varphi \,\partial_2v_N^+-[\partial_1\widehat{H}_{\tau}]R\partial_2\varphi -[\partial_2\partial_1\widehat{H}_{\tau}]\varphi\,\partial_2v_2^+ \\
+ \big([\partial_2\partial_1\hat{p}]+\widehat{H}_2^+[\partial_2\partial_1\widehat{H}_{\tau}]\big)\varphi\,\partial_2v_N^+ +\mathcal{Q}_2 \\[3pt]
=\underbrace{[\partial_1\hat{p}]\,\partial_t\partial_2\varphi\, \partial_2\varphi}  + \underline{{\rm coeff}\,\partial_0H_N^+\,\partial_2\varphi|_{x_1=0} }  \\ + \Bigl( {\rm coeff}\,\partial_2^2\varphi\, \partial_2\varphi +
{\rm coeff}\,( \partial_2\varphi )^2+{\rm coeff}\,\varphi\, \partial_2\varphi +{\rm coeff}\,v_N^+ \partial_2\varphi \\+{\rm coeff}\,v_2^+\partial_2\varphi +{\rm coeff}\,H_N^+\partial_2\varphi +{\rm coeff}\,\varphi\, \partial_2v_N^+  + {\rm coeff}\,\varphi \,\partial_2v_2^+ \\
+{\rm coeff}\,v_2^+\varphi + {\rm coeff}\,[H_1] \partial_2v_2^+ + {\rm coeff}\,[H_1] \partial_2v_N^+
 \\+ {\rm coeff}\,v_2^+ [H_1]+ {\rm coeff}\,F_N^+\partial_2\varphi +{\rm coeff}\, F_N^+ \varphi\Bigr)\Bigr|_{x_1=0}.
\label{Q2"}
\end{multline}

The underbraced term in \eqref{Q2"} is {\it  most important}  because under the Rayleigh-Taylor sign condition \eqref{RTL} it gives us the control on the $L^2$ norm of $\partial_2\varphi$ (see below). Having in hand this control, only the underlined term in \eqref{Q2"} needs an additional care whereas the rest terms in \eqref{Q2"} can be easily treated by the integration by parts, etc. Indeed,
\begin{multline}
2 \int_{\partial\Omega_t}Q_2dx_2ds=\int_{\mathbb{R}}[\partial_1\hat{p}](t)(\partial_2\varphi )^2(t)\,dx_2  -K(t) -M(t) \\
\geq \epsilon \, \|\partial_2 \varphi (t)\|^2_{L^2(\mathbb{R})}-K(t) -M(t),
\label{IntQ2}
\end{multline}
where
\[
K(t)=\int_{\partial\Omega_t}{\rm coeff}\,\partial_0H_N^+\,\partial_2\varphi|_{x_1=0}dx_2ds
\]
and
\begin{equation}
M(t)\leq C\biggl\{ \sum_\pm \|F^\pm \|^2_{H^1(\Omega_T)}   +\int_0^tI (s)ds \\+ \|\varphi \|^2_{L^2(\partial\Omega_t)} +\|\partial_2 \varphi\|^2_{L^2(\partial\Omega_t)}\biggr\}.
\label{Mt}
\end{equation}
Here we used the integration by parts, the elementary inequality
\begin{equation}
\|U^\pm _{|x_1=0} (t)\|^2_{L^2(\partial\Omega_t)}\leq \|U^\pm  (t)\|^2_{L^2(\Omega_t)} +\|\partial_1 U^\pm  (t)\|^2_{L^2(\Omega_t)},
\label{tr}
\end{equation}
estimate \eqref{d1U} and the trace theorem for $F_N^+$ (or even the inequality like \eqref{tr}).

It remains to estimate the boundary integral $K(t)$ of the underlined term in \eqref{Q2"}. To this end we first integrate by parts and use the boundary condition \eqref{b50b.4}:
\begin{multline*}
K(t)=L(t) +\int_{\partial\Omega_t}\bigl({\rm coeff}\,H_N^+\,\partial_2\varphi +{\rm coeff}\,H_N^+\,\partial_2(\partial_0\varphi)\bigr)\bigr|_{x_1=0}dx_2ds \\ = L(t)+
\int_{\partial\Omega_t}\bigl({\rm coeff}\,H_N^+\,\partial_2\varphi +{\rm coeff}\,H_N^+\,\varphi +{\rm coeff}\,H_N^+\,\partial_2v_N^+\bigr)\bigr|_{x_1=0}dx_2ds ,
\end{multline*}
where
\[
L(t)=\int_{\mathbb{R}}{\rm coeff}\,H_N^+\,\partial_2\varphi|_{x_1=0}dx_2.
\]
Then we apply \eqref{tr}, estimate \eqref{d1U} and arguments as in \eqref{intbypart} with $\alpha =2$:
\begin{equation}
K(t) \leq L(t) + C\biggl\{ \sum_\pm \|F^\pm \|^2_{H^1(\Omega_T)} +\int_0^tI (s)ds \\ + \|\varphi \|^2_{L^2(\partial\Omega_t)} +\|\partial_2 \varphi\|^2_{L^2(\partial\Omega_t)}\biggr\}.
\label{Kt}
\end{equation}

We now estimate the integral $L(t)$ by using the Young inequality, the passage to the volume integral, the relation $\partial_1H_N^+={\rm coeff}\,H_2^++{\rm coeff}\,\partial_2H_2^++F_7^+$ following from \eqref{b43b}, estimate \eqref{f7"}, and the elementary inequality \eqref{elem'}:
\begin{multline}
 L(t) \leq C \left\{ \tilde{\varepsilon}\|\partial_2 \varphi (t)\|^2_{L^2(\mathbb{R})} +\frac{1}{\tilde{\varepsilon}}\int_{\mathbb{R}}(H_N^+)^2\bigr|_{x_1=0}dx_2\right\}\\
 = C \biggl\{ \tilde{\varepsilon}\|\partial_2 \varphi (t)\|^2_{L^2(\mathbb{R})}  -\frac{2}{\tilde{\varepsilon}}\int_{\mathbb{R}^2_+}H_N^+\partial_1H_N^+ dx\biggr\}\\
 \leq C \biggl\{ \tilde{\varepsilon}\|\partial_2 \varphi (t)\|^2_{L^2(\mathbb{R})}  +\frac{1}{\tilde{\varepsilon}}  \biggl(\frac{1}{\tilde{\varepsilon}^2} \|H_N^+(t)\|^2_{L^2(\mathbb{R}^2_+)}\\ +\tilde{\varepsilon}^2\bigl( \| H_2^+(t)\|^2_{L^2(\mathbb{R}^2_+)} +\|\partial_2H_2^+(t)\|^2_{L^2(\mathbb{R}^2_+)} + \|F_7^+(t)\|^2_{L^2(\mathbb{R}^2_+)}\bigr)\biggr)\biggr\} \\
 \leq \tilde{\varepsilon}C \left\{ \|\partial_2 \varphi (t)\|^2_{L^2(\mathbb{R})} +\|\partial_2U^+(t)\|^2_{L^2(\mathbb{R}^2_+)}\right\} + \widetilde{C}(\tilde{\varepsilon})\left\{ \|F^+ \|^2_{H^1(\Omega_T)}+ \int_0^tI (s)ds\right\},
\label{Lt}
\end{multline}
where we may consider the same $\tilde{\varepsilon}$ as in \eqref{estd0}. Then,
\eqref{q2}, \eqref{IntQ2}, \eqref{Mt}, \eqref{Kt} and \eqref{Lt} imply
\begin{equation} \begin{split}
&\sum_\pm \| \partial_2U^\pm (t)\|^2_{L^2(\mathbb{R}^2_+)}  + \|\partial_2\varphi (t)\|^2_{L^2(\mathbb{R})}\\
\leq  &\tilde{\varepsilon}C\left\{ \|\partial_2 \varphi (t)\|^2_{L^2(\mathbb{R})} +\|\partial_2U^+(t)\|^2_{L^2(\mathbb{R}^2_+)}\right\} + \widetilde{C}(\tilde{\varepsilon})\biggl\{\sum_\pm \|F^\pm \|^2_{H^1(\Omega_T)}\\
 & +\int_0^tI (s)ds+ \|\varphi \|^2_{L^2(\partial\Omega_t)} +\|\partial_2 \varphi\|^2_{L^2(\partial\Omega_t)}\biggr\}.
\label{estd2}
\end{split}   \end{equation}

At last, choosing $\tilde{\varepsilon} $ to be small enough and taking into account that
\[
\|\partial_0U^\pm (t)\|^2_{L^2(\mathbb{R}^2_+)} +c_1\|\partial_2U^\pm (t)\|^2_{L^2(\mathbb{R}^2_+)} \geq c_2\left\{ \|\partial_tU^\pm (t)\|^2_{L^2(\mathbb{R}^2_+)} +\|\partial_2U^\pm (t)\|^2_{L^2(\mathbb{R}^2_+)}\right\},
\]
with suitable positive constants $c_1$ and $c_2$ depending on the basic state,  the combination of \eqref{estd0} and \eqref{estd2} yields the energy inequality
\begin{equation}
J(t) \leq C \left\{ \sum_\pm \|F^\pm \|^2_{H^1(\Omega_T)}  +\int_0^tJ (s)ds\right\},
\label{main_energy_in}
\end{equation}
where
\[
J(t)=I(t) + \|\varphi (t) \|^2_{L^2(\mathbb{R})} +\|\partial_2 \varphi (t)\|^2_{L^2(\mathbb{R})}.
\]
Then, applying Gronwall's lemma and integrating in time over $(-\infty,T)$, from \eqref{main_energy_in} we derive the energy a priori estimate
\begin{multline}
\sum_{\pm}\left(
\|U^\pm\|_{L^2(\Omega_T)}+
\|\partial_t U^\pm\|_{L^2(\Omega_T)}+ \|\partial_2U^\pm\|_{L^2(\Omega_T)}\right) \\
+\|\varphi\|_{L^2(\partial\Omega_T)} + \|\partial_2\varphi\|_{L^2(\partial\Omega_T)}\leq C\sum_{\pm}\|F^\pm \|_{H^{1}(\Omega_T)}.
\label{main_estim}
\end{multline}

In view of \eqref{tr}, from the boundary condition \eqref{b50b.4} we obtain the estimate
\begin{equation}
\|\partial_t\varphi\|_{L^2(\partial\Omega_T)}  \leq C \bigl\{\|U^+\|_{L^2(\Omega_T)}+
\|\partial_1 U^+\|_{L^2(\Omega_T)}  +\|\varphi\|_{L^2(\partial\Omega_T)} + \|\partial_2\varphi\|_{L^2(\partial\Omega_T)}\bigr\}.
\label{front_t}
\end{equation}
Combining \eqref{d1U} (with $t=T$), \eqref{main_estim} and \eqref{front_t}, we deduce the desired a priori estimate \eqref{main_est'}. Thus, the proof of Lemma \ref{lem1} is complete. Recall that Lemma \ref{lem1} implies estimate \eqref{main_est} in Theorem \ref{t1}.

\section{Well-posedness of the linearized problem}
\label{sec:exist}

\subsection{Existence of solutions for a fully noncharacteristic ``strictly dissipative'' re\-gularization} We prove the existence of a unique solution $((\dot{U}^+,\dot{U}^- ),\varphi )\in H^1(\Omega_T)\times H^1(\partial\Omega_T)$ to problem \eqref{b48^}--\eqref{b51^} by using its fully noncharacteristic ``strictly dissipative'' regularization containing a small parameter of regularization $\varepsilon$:
\begin{multline}
A_0(\widehat{U}^{\pm})\partial_t\dot{U}^{\pm \varepsilon} -\varepsilon\partial_1\dot{U}^{\pm \varepsilon} +\widetilde{A}_1(\widehat{U}^{\pm},\widehat{\Psi}^{\pm})\partial_1\dot{U}^{\pm \varepsilon}\\+A_2(\widehat{U}^{\pm} )\partial_2\dot{U}^{\pm \varepsilon}  +\mathcal{C}(\widehat{U}^{\pm},\widehat{\Psi}^{\pm})\dot{U}^{\pm \varepsilon} =f^{\pm}\qquad \mbox{in}\ \Omega_T,\label{r48^}
\end{multline}
\begin{equation}
\left(
\begin{array}{c}
\dot{p}^{+ \varepsilon}-\dot{p}^{- \varepsilon} + \varphi^{\varepsilon} [\partial_1\hat{p}]\\[3pt]
\dot{v}_1^{+ \varepsilon}-\dot{v}_1^{- \varepsilon}\\[3pt]
\dot{v}_2^{+ \varepsilon}-\dot{v}_2^{- \varepsilon}\\[3pt]
\dot{H}_{\tau}^{+ \varepsilon}-\dot{H}_{\tau}^{- \varepsilon}+ \varphi^{\varepsilon} [\partial_1\widehat{H}_{\tau}]\\[3pt]
\partial_t\varphi^{\varepsilon} +\hat{v}_2^{+}\partial_2\varphi^{\varepsilon} -\dot{v}_{N}^{+ \varepsilon} - \varphi^{\varepsilon} \partial_1\hat{v}_N^+
\end{array}
\right)=g \qquad \mbox{on}\ \partial\Omega_T,
\label{r50^}
\end{equation}
\begin{equation}
(\dot{U}^{+ \varepsilon},\dot{U}^{- \varepsilon},\varphi^{\varepsilon} )=0\qquad \mbox{for}\ t<0.\label{r51^}
\end{equation}
The boundary conditions \eqref{r50^} for $(\dot{U}^{+ \varepsilon},\dot{U}^{- \varepsilon},\varphi^{\varepsilon} )$ coincide with \eqref{b50^} whereas the interior equations \eqref{r48^} contain the additional terms $-\varepsilon\partial_1\dot{U}^{\pm \varepsilon} $ compared to the original (not regularized) equations \eqref{b48^}.

\begin{lemma}
Let assumptions \eqref{a5}--\eqref{cdass} be satisfied and the constant $\varepsilon >0$ be small enough compared to the constant $\bar{\kappa}_0$ in \eqref{cdass}. Then, for any fixed constant $\varepsilon$ and for all $f^{\pm} \in H^1(\Omega_T)$ and $g\in H^{1}(\partial\Omega_T)$ which vanish in the past, problem \eqref{r48^}--\eqref{r51^} has a unique solution $((\dot{U}^{+ \varepsilon}, \dot{U}^{- \varepsilon}) ,\varphi^{\varepsilon} )\in H^1(\Omega_T)\times H^1(\partial\Omega_T)$.
\label{lem2}
\end{lemma}

\begin{proof}
Since the constant $\varepsilon >0$ is small enough compared to the constant $\bar{\kappa}_0$ in \eqref{cdass}, the number of incoming characteristics for systems \eqref{r48^} coincides with that for systems \eqref{b48^}. Considering for a moment ${\varphi}^{\,\varepsilon}$ as a given function and taking into account \eqref{quadr}, we easily see that the boundary conditions \eqref{r50^} (without the last one) are {\it strictly dissipative}:
\[
\sum_{\pm}\bigl(\bigl\{\varepsilon I_6- \widetilde{A}_1(\widehat{U}^\pm,\widehat\Psi^{\pm}\ddot{})\bigr\}\dot{U}^{\pm \varepsilon} , \dot{U}^{\pm \varepsilon} \bigr)\bigr|_{x_1=0}\geq
\sum_{\pm}\frac{\varepsilon}{2}|\dot{U}^{\pm \varepsilon}_{|x_1=0}|^2- \frac{C}{\varepsilon}(|{g}|^2+|{\varphi}^{\,\varepsilon}|^2),
\]
where $I_6$ is the unit matrix of order 6.  Then, by standard arguments of the energy method we obtain
\begin{multline}
\sum_{\pm}\bigl(\|\dot{U}^{\pm \varepsilon} (t)\|^2_{L^{2}(\mathbb{R}^2_+)} +\|\dot{U}^{\pm \varepsilon}_{|x_1=0} \|^2_{L^{2}(\partial\Omega_t)} \bigr) \\ \leq \frac{C}{\varepsilon^2} \biggl\{\sum_{\pm}\big(\|f^\pm \|^2_{L^{2}(\Omega_T)} + \|\dot{U}^{\pm \varepsilon} \|^2_{L^{2}(\Omega_t)}\big) + \|{g}\|^2_{L^{2}(\partial\Omega_T)}
+\|{\varphi}^{\,\varepsilon} \|^2_{L^2(\partial\Omega_t)}\biggr\}.
\label{est_reg}
\end{multline}

Using the Young inequality, from the last boundary condition in \eqref{r50^} we easily deduce
\begin{equation}
\|{\varphi}^{\,\varepsilon} (t)\|^2_{L^{2}(\mathbb{R}^)} \leq C\Big( \delta\|\dot{U}^{\pm \varepsilon}_{|x_1=0} \|^2_{L^{2}(\partial\Omega_t)} +\frac{1}{\delta}\|{\varphi}^{\,\varepsilon} \|^2_{L^2(\partial\Omega_T)}\Big),
\label{est_reg^}
\end{equation}
with a positive constant $\delta$. Combining inequalities \eqref{est_reg} and \eqref{est_reg^}, for a sufficiently small $\delta$ we obtain
\begin{multline*}
\sum_{\pm}\|\dot{U}^{\pm \varepsilon} (t)\|^2_{L^{2}(\mathbb{R}^2_+)} +\|{\varphi}^{\,\varepsilon} (t)\|^2_{L^{2}(\mathbb{R}^)}   \\ \leq \frac{C}{\varepsilon^2} \biggl\{\sum_{\pm}\big(\|f^\pm \|^2_{L^{2}(\Omega_T)} + \|\dot{U}^{\pm \varepsilon} \|^2_{L^{2}(\Omega_t)}\big)  + \|{g}\|^2_{L^{2}(\partial\Omega_T)}
+\|{\varphi}^{\,\varepsilon} \|^2_{L^2(\partial\Omega_t)}\biggr\}.
\end{multline*}
Applying then Gronwall's lemma, we get the $L^2$ estimate
\begin{equation}
\sum_{\pm}\|\dot{U}^{\pm \varepsilon} \|_{L^{2}(\Omega_T)} +\|{\varphi}^{\,\varepsilon} \|_{L^2(\partial\Omega_T)}   \leq {C}({\varepsilon} ) \biggl\{\sum_{\pm}\|f^\pm \|_{L^{2}(\Omega_T)} +  \|{g}\|_{L^{2}(\partial\Omega_T)}\biggr\},
\label{est_reg_a}
\end{equation}
where $C(\varepsilon )\rightarrow +\infty$ as $\varepsilon\rightarrow 0$, but now the constant $\varepsilon$ is {\it fixed}. Note that we could also include the $L^2$ norm of the trace of solution in this estimate.
Using tangential differentiation and the fact that the boundary $x_1=0$ is not characteristic for $\varepsilon \neq 0$, we also easily obtain the $H^1$ estimate
\begin{equation}
\sum_{\pm}\|\dot{U}^{\pm \varepsilon} \|_{H^{1}(\Omega_T)} +\|{\varphi}^{\,\varepsilon} \|_{H^1(\partial\Omega_T)}   \\  \leq {C}({\varepsilon}) \biggl\{\sum_{\pm}\|f^\pm \|_{H^{1}(\Omega_T)} +  \|{g}\|_{H^{1}(\partial\Omega_T)}\biggr\}.
\label{est_reg'}
\end{equation}

Having in hand the $L^2$ estimate \eqref{est_reg_a}, the existence of a weak $L^2$ solution to problem \eqref{r48^}--\eqref{r51^} can be obtained by the classical duality argument. To this end we should obtain an $L^2$ a priori estimate for the dual problem. The boundary conditions \eqref{r50^} are not strictly dissipative in the classical sense but in some sense they are similar to them and it is natural to expect the same from the dual problem. This will enable us to derive an a priori estimate for the dual problem.

To avoid overloading the paper and hiding simple ideas in the shadow of unimportant technical calculations, for the dual problem we restrict ourselves to ``frozen'' coefficients and the case $\hat{\varphi}=0$. More precisely, we consider the following constant coefficients version of problem \eqref{r48^}--\eqref{r51^} (we also drop the superscript $\varepsilon$ by the unknowns):
\begin{equation}
L^{\pm}U^{\pm}:=A_0\partial_t{U}^{\pm} -\varepsilon\partial_1{U}^{\pm} \pm{A}_1\partial_1{U}^{\pm}+A_2\partial_2{U}^{\pm} =f^{\pm}\qquad \mbox{in}\ \Omega_T,\label{mod1}
\end{equation}
\begin{equation}
[p]= a\varphi,\quad [v]=0,\quad [H_2]=b\varphi,\quad \partial_0\varphi =v_1^+ \qquad \mbox{on}\ \partial\Omega_T,
\label{mod2}
\end{equation}
\begin{equation}
({U}^+,U^-,\varphi )=0\qquad \mbox{for}\ t<0,\label{mod3}
\end{equation}
where $A_{\alpha}$ are the matrices of the MHD system \eqref{4} calculated on a constant vector with $\hat{v}_1=0$, $v_2=\hat{v}_2$, $H_1=\widehat{H}_1$, etc. ($\hat{v}_2=\hat{v}_2^\pm ={\rm const}$, $\widehat{H}_1 =\widehat{H}_1^\pm ={\rm const}$, etc.);  $a=-[\partial_1\hat{p}]={\rm const}$,  $b=-[\partial_1\widehat{H}_2]={\rm const}$, $\partial_0=\partial_t+\hat{v}_2\partial_2$, and we consider yet homogeneous boundary conditions.

The dual operators have the form $(L^\pm )^*=-L^\pm$ (recall that we consider constant coefficients) and the boundary conditions for the dual problem are defined from the requirement
\begin{multline*}
\sum_{\pm}\Big\{
(L^{\pm}U^{\pm},\overline{U}^{\pm})_{L^2(\Omega_T)}-(U^{\pm},(L^{\pm})^*\overline{U}^{\pm})_{L^2(\Omega_T)}\Big\} \\= -\big[(A_1U,\overline{U})_{L^2(\Omega_T)}\big]+\varepsilon \sum_{\pm}(U^\pm ,\overline{U}^\pm )_{L^2(\Omega_T)}=0,
\end{multline*}
with $\overline{U}^{\pm}|_{t=T}=0$ and $U^\pm$ satisfying the boundary conditions \eqref{mod2}. Omitting simple calculations, we write down these boundary conditions:
\begin{eqnarray*}
\partial_0w=(a+b\widehat{H}_2)\overline{v}_1^+
-\varepsilon a \overline{p}^+-b( \widehat{H}_1\overline{v}_2^+ +\varepsilon \overline{H}_2^+), & \\[3pt]
 [\overline{v}_1]=\varepsilon \langle \overline{p}\rangle , & \\[3pt]
 \widehat{H}_1 [\overline{H}_2]=-\varepsilon \langle \overline{v}_1\rangle , &  \\[3pt]
  [\widehat{H}_2 \overline{v}_1 -\widehat{H}_1\overline{v}_2]=\varepsilon \langle \overline{H}_2\rangle , &  \\[3pt]
  \overline{H}_1^+ = \overline{H}_1^- =\overline{S}^+=\overline{S}^- =0& \qquad \mbox{at}\ x_1=0,
\end{eqnarray*}
where $\langle z\rangle :=(z^+ +z^-)|_{x_1=0}$ and
\begin{equation}
w=[\overline{p}+\widehat{H}_2\overline{H}_2]-\varepsilon \langle \overline{v}_1\rangle . \label{ww}
\end{equation}
Changing the time $t$ to $-t$ and then shifting it to the value $T$, we obtain the dual problem in the same form as the original problem \eqref{mod1}--\eqref{mod3}:
\begin{equation}
A_0\partial_t\overline{U}^{\pm} +\varepsilon\partial_1\overline{U}^{\pm} \mp {A}_1\partial_1\overline{U}^{\pm}-A_2\partial_2\overline{U}^{\pm} =\overline{f}^{\pm}\qquad \mbox{in}\ \Omega_T,\label{mod4}
\end{equation}
\begin{eqnarray}
\overline{\partial}_0w=\varepsilon a \overline{p}^+-(a+b\widehat{H}_2)\overline{v}_1^+
 +b( \widehat{H}_1\overline{v}_2^+ +\varepsilon \overline{H}_2^+), & \label{mod5.1} \\[3pt]
 [\overline{v}_1]=\varepsilon \langle \overline{p}\rangle , & \label{mod5.2} \\[3pt]
 \widehat{H}_1 [\overline{H}_2]=-\varepsilon \langle \overline{v}_1\rangle , &  \label{mod5.3} \\[3pt]
  [\widehat{H}_2 \overline{v}_1 -\widehat{H}_1\overline{v}_2]=\varepsilon \langle \overline{H}_2\rangle , &  \label{mod5.4} \\[3pt]
  \overline{H}_1^+ = \overline{H}_1^- =\overline{S}^+=\overline{S}^- =0& \qquad \mbox{on}\ \partial\Omega_T, \label{mod5.5}
\end{eqnarray}
\begin{equation}
(\overline{U}^+,\overline{U}^-)=0\qquad \mbox{for}\ t<0,\label{mod6}
\end{equation}
where $\overline{\partial}_0=\partial_t-\hat{v}_2\partial_2$. The hyperbolic systems \eqref{mod4} together have four additional incoming characteristics compared to systems \eqref{mod1}. That is, the number of boundary conditions in \eqref{mod5.1}--\eqref{mod5.5} (they are eight together) coincides with the number of incoming characteristics of systems \eqref{mod4}.

We introduce the notation
\[
V^\pm =(\overline{p}^\pm ,\overline{v}_1^\pm , \overline{v}_2^\pm ,\overline{H}_2^\pm ).
\]
Then, in view of \eqref{ww} and the boundary conditions \eqref{mod5.2}--\eqref{mod5.4}, using the fact that $\langle z\rangle =2z^+_{|x_1=0} -[z]$,  we obtain
\begin{equation}
[V]=2\varepsilon
\begin{pmatrix}
\overline{v}_1^+ +(\widehat{H}_2/\widehat{H}_1)\,\overline{v}_2^+ \\[3pt]
\overline{p}^+ \\[3pt]
(\widehat{H}_2/\widehat{H}_1)\,\overline{p}^+ - (1/\widehat{H}_1)\,\overline{H}_2^+\\[3pt]
-(1/\widehat{H}_1)\,\overline{v}_2^+
\end{pmatrix}
+\varepsilon^2 B_1V^+ + w\,B_2V^+,
\label{du1}
\end{equation}
where the coefficient matrix $B$ has elements of order ${\cal O}(1)$ as $\varepsilon \rightarrow 0$ whose exact form is of no interest and the exact form of the coefficient matrix $B_2$ is also unimportant for subsequent calculations.

As usual, the quadratic form
\[
Q=[( A_1\overline{U} , \overline{U} )]  -\varepsilon \langle |\overline{U}|^2 \rangle  = 2[\overline{p}\,\overline{v}_1]+2\widehat{H}_2[\overline{v}_1\overline{H}_2]-2\widehat{H}_1[\overline{v}_2\overline{H}_2] -\varepsilon \langle |V|^2 \rangle
\]
with the boundary matrix plays the crucial role for deriving a priori estimates by the energy method. Here we have already used the boundary conditions \eqref{mod5.5}. Applying then \eqref{du1} and omitting simple algebraic calculations, we get
\begin{multline}
Q=2\varepsilon |V^+_{|x_1=0}|^2 +\varepsilon^2 (MV^+,V^+)|_{x_1=0}\,+ w\, (q,V^+)|_{x_1=0} \\
= \varepsilon \langle |V|^2 \rangle +\varepsilon^2 (\widetilde{M}V^+,V^+)|_{x_1=0}\,+ w\, (\tilde{q},V^+)|_{x_1=0},
\label{QQQ}
\end{multline}
where the coefficient matrices $M$ and $\widetilde{M}$ as well the coefficient vectors $q$ and $\tilde{q}$ are of no interest (the elements of these matrices are of order ${\cal O}(1)$ for small $\varepsilon$, i.e., between the first two terms in the right-hand side of \eqref{QQQ} the first positive one is leading).

Using \eqref{QQQ} and the Young inequality, by standard arguments of the energy method we obtain for systems \eqref{mod4} (for $\varepsilon$ small enough) the energy inequality
\begin{multline}
\sum_{\pm}\big(\|\overline{U}^\pm  (t)\|^2_{L^{2}(\mathbb{R}^2_+)} +\|\overline{U}^\pm _{|x_1=0} \|^2_{L^{2}(\partial\Omega_t)}\big)  \\ \leq {C}({\varepsilon})
\Big( \|\overline{f}\|^2_{L^{2}(\Omega_T)} + \sum_{\pm}\|\overline{U}^\pm \|^2_{L^{2}(\Omega_t)}+ \|w|_{x_1=0} \|^2_{L^{2}(\partial\Omega_t)}
\Big)
\label{mod8}
\end{multline}
(note that, in view of \eqref{mod5.5}, $\|\overline{V}^\pm _{|x_1=0} \|^2_{L^{2}(\partial\Omega_t)}=\|\overline{U}^\pm _{|x_1=0} \|^2_{L^{2}(\partial\Omega_t)}$). At the same time, from the boundary condition \eqref{mod5.1} we get
\begin{equation}
\|w|_{x_1=0} (t)\|^2_{L^{2}(\mathbb{R})}\leq
C\Big( \delta\|V^+|_{x_1=0} \|^2_{L^{2}(\partial\Omega_t)} +\frac{1}{\delta}\|w|_{x_1=0} \|^2_{L^{2}(\partial\Omega_t)} \Big),
\label{mod9}
\end{equation}
with a positive constant $\delta$. By choosing $\delta$ small enough and combining inequalities \eqref{mod8} and \eqref{mod9}, one gets
\begin{multline*}
\sum_{\pm}\|\overline{U}^\pm (t)\|^2_{L^{2}(\mathbb{R}^2_+)} +\|w |_{x_1=0}(t)\|^2_{L^{2}(\mathbb{R})} \\ \leq {C}({\varepsilon})
\Big( \|\overline{f}\|^2_{L^{2}(\Omega_T)} + \sum_{\pm}\|\overline{U}^\pm \|^2_{L^{2}(\Omega_t)}+ \|w|_{x_1=0} \|^2_{L^{2}(\partial\Omega_t)}
\Big).
\end{multline*}
Applying Gronwall's lemma to the last inequality gives the desired a priori estimate
\begin{equation}
\sum_{\pm}\|\overline{U}^\pm \|_{L^{2}(\Omega_T)}  \leq C({\varepsilon}) \sum_{\pm}\|\overline{f}^\pm \|_{L^{2}(\Omega_T)}
\label{estdu}
\end{equation}
for the dual problem.

Thanks to the control on the trace of solution (``strict dissipativity'') we can easily generalize estimate \eqref{estdu} to the case of inhomogeneous boundary conditions in \eqref{mod5.1}--\eqref{mod5.5}. Thus, for the problem adjoint to \eqref{r48^}--\eqref{r51^} we can derive the a priori $L^2$ estimate
\begin{equation}
\sum_{\pm}\|\overline{U}^{\pm \varepsilon} \|_{L^{2}(\Omega_T)}   \leq {C}({\varepsilon}) \biggl\{\sum_{\pm}\|\overline{f}^\pm \|_{L^{2}(\Omega_T)} +  \|\overline{g}\|_{L^{2}(\partial\Omega_T)}\biggr\}
\label{est_reg_b}
\end{equation}
with no loss of derivatives from the data $\overline{f}^\pm$ and $\overline{g}$. In view of the a priori estimates \eqref{est_reg_a} and \eqref{est_reg_b}, the classical duality argument gives the existence of a weak $L^2$ solution to problem \eqref{r48^}--\eqref{r51^}. Then, tangential differentiation and the fact the boundary $x_1=0$ is not characteristic for $\varepsilon \neq 0$ gives the existence of an $H^1$ solution for any fixed sufficiently small parameter $\varepsilon >0$. Its uniqueness follows from the a priori estimate \eqref{est_reg_a}.
\end{proof}

The a priori estimate \eqref{est_reg'} is not uniform in $\varepsilon$ and not suitable to pass to the limit as $\varepsilon\rightarrow 0$. However, thanks to Lemma \ref{lem2} we have now the existence of solutions of the regularized problem for any fixed sufficiently small parameter $\varepsilon >0$. Estimate \eqref{est_reg'} as well as Lemma \ref{lem2} holds true even if the stability condition \eqref{RTL} is violated. Below, assuming the fulfilment of \eqref{RTL} and a more regularity for the source term $g$ (as in Theorem \ref{t1}, $g\in H^{3/2}(\partial\Omega_T)$), we will get for problem \eqref{r48^}--\eqref{r51^} an a priori estimate uniform in $\varepsilon$. Actually, this estimate is nothing else than our basic a priori estimate \eqref{main_est}.

\subsection{Uniform in $\varepsilon$ estimate and passage to the limit}
We now revisit several places in the arguments of Section \ref{sec:4} and adapt them for problem \eqref{r48^}--\eqref{r51^}. First of all, as for problem \eqref{b48^}--\eqref{b51^}, it is enough to prove the a priori estimate \eqref{main_est'} for a corresponding reduced problem with homogeneous boundary conditions. Passing to the new unknown
\[
U^{\pm \varepsilon}=\dot{U}^{\pm \varepsilon}-\widetilde{U}^\pm ,
\]
where $\widetilde{U}^\pm$ is the same as in \eqref{a87'}, we obtain the following problem (for the sake of notational convenience, we below omit the index $\varepsilon$ in $U^{\pm \varepsilon}$, $F^{\pm \varepsilon}$, etc.):
\begin{multline}
A_0(\widehat{U}^{\pm})\partial_t{U}^{\pm} +\bigl(\widetilde{A}_1(\widehat{U}^{\pm},\widehat{\Psi}^{\pm})-\varepsilon I_6\bigr)\partial_1{U}^{\pm}\\ +A_2(\widehat{U}^{\pm} )\partial_2{U}^{\pm}
 +\mathcal{C}(\widehat{U}^{\pm},\widehat{\Psi}^{\pm}){U}^{\pm} =F^{\pm}\qquad \mbox{in}\ \Omega_T,\label{rb48b}
\end{multline}
\begin{eqnarray}
 {[}p{]}=- \varphi {[}\partial_1\hat{p}{]}, & \label{rb50b.1}\\[3pt]
  {[}v{]}=0, & \label{rb50b.2}\\[3pt]
  {[}H_{\tau}{]}=-\varphi {[}\partial_1\widehat{H}_{\tau}{]}, & \label{rb50b.3} \\[3pt]
{v}_{N}^+ = \partial_0\varphi - \varphi \partial_1\hat{v}_N^+  & \qquad \mbox{on}\ \partial\Omega_T
\label{rb50b.4}
\end{eqnarray}
\begin{equation}
({U}^+,{U}^-,\varphi )=0\qquad \mbox{for}\ t<0,\label{rb51b}
\end{equation}
where
\[
F^\pm =f^\pm-A_0(\widehat{U}^{\pm})\partial_t\widetilde{U}^{\pm} -\left(\widetilde{A}_1(\widehat{U}^{\pm},\widehat{\Psi}^{\pm})-\varepsilon I_6\right)\partial_1\widetilde{U}^{\pm} \\ -A_2(\widehat{U}^{\pm} )\partial_2\widetilde{U}^{\pm}-\mathcal{C}(\widehat{U}^{\pm},\widehat{\Psi}^{\pm})\widetilde{U}^\pm
\]
obeys the uniform in $\varepsilon$  estimate \eqref{nonhom_F'}.

Solutions to problem \eqref{rb48b}--\eqref{rb51b} satisfy \eqref{b43b} and the counterpart of \eqref{aa1^} reads:
\begin{equation}
\partial_t a^{\pm}-\varepsilon\partial_1a^{\pm}+  \frac{1}{\partial_1\widehat{\Phi}^{\pm}}\left\{ (\hat{w}^{\pm} ,\nabla a^{\pm}) + a^{\pm}\bigl({\rm div}\,\hat{u}^{\pm}-\varepsilon\partial_1^2\widehat{\Phi}^{\pm}\bigr)\right\}={\mathcal F}^{\pm}\quad \mbox{in}\ \Omega_T,
\label{raa1^}
\end{equation}
where $a^{\pm}=F_7^{\pm}/\partial_1\widehat{\Phi}^{\pm},\quad {\mathcal F}^{\pm}=({\rm div}\,{F}_{h}^{\pm})/\partial_1\widehat{\Phi}^{\pm}$,
${F}_{h}^{\pm}=(F_{N}^{\pm} ,\partial_1\widehat{\Phi}^{\pm}F_5^{\pm})$ and $F_{N}^{\pm}=F_4^{\pm}-F_5^{\pm}\partial_2\widehat{\Psi}^{\pm}$. As \eqref{aa1^}, equations \eqref{raa1^} do not still need any boundary conditions and we get estimates \eqref{f7^} and \eqref{f7"}. Moreover, below we will also need a uniform in $\varepsilon$ estimate for the trace $\varepsilon F_7^+|_{x_1=0}$. Exactly as for the case of strictly dissipative boundary conditions, from \eqref{raa1^} we easily obtain the inequality
\[
\|a^+(t)\|^2_{L^2(\mathbb{R}^2_+)} +\varepsilon\|a^+_{|x_1=0}\|^2_{L^2(\partial\Omega_t)}\leq C\left\{\|F^+ \|^2_{H^1(\Omega_T)} + \|a^+\|^2_{L^2(\Omega_T)}\right\}
\]
giving, together with \eqref{f7^}, the estimate
\[
\sqrt{\varepsilon}\|a^+_{|x_1=0}\|_{L^2(\partial\Omega_T)}\leq C\|F^+ \|_{H^1(\Omega_T)}
\]
and
\begin{equation}
\|\varepsilon F_7^+|_{x_1=0}\|_{L^2(\partial\Omega_T)}\leq \sqrt{\varepsilon}C\|a^+_{|x_1=0}\|_{L^2(\partial\Omega_T)}\leq C\|F^+ \|_{H^1(\Omega_T)}.
\label{F7tr}
\end{equation}

In fact, in the proof of estimate \eqref{main_est'} for problem \eqref{rb48b}--\eqref{rb51b} only \eqref{eqH} needs a special care whereas the rest arguments are the same as in Section \ref{sec:4} for problem \eqref{b48b}--\eqref{b51b}. Indeed, taking into account \eqref{b43b}, for sufficiently small $\varepsilon$ we still can resolve systems \eqref{rb48b} for the normal derivatives $\partial_1V^\pm$ without the appearance of the unwanted multiplier $1/\varepsilon$. The rest of the arguments towards the proof of estimates \eqref{d1U} and \eqref{1Vt} are the same as in the proof of Proposition \ref{p5}. In particular, the counterparts of equations \eqref{entropy} will contain the additional terms $-\varepsilon\partial_1S^\pm$ in the left-hand sides of the equations. But, such equations for $S^\pm$ do not still need any boundary conditions and we get \eqref{ent_est}.

In the left-hand side of the counterpart of the energy inequality \eqref{q0} for problem \eqref{rb48b}--\eqref{rb51b} we have the additional positive terms
\[
\varepsilon\sum_{\pm}\|\partial_0U^\pm _{|x_1=0}\|^2_{L^2(\partial\Omega_t)}
\]
which can be just thrown away to make the energy inequality rougher. The rest of the arguments towards the proof of estimate \eqref{estd0} are again the same as in Section \ref{sec:4}, and it is important that the constant $C$ in \eqref{estd0} for \eqref{rb48b}--\eqref{rb51b}  does not depend on $\varepsilon$.

Throwing away most of the positive terms containing the multiplier $\varepsilon$ in the counterpart of \eqref{q2} we obtain inequality \eqref{q2} with the quadratic form $Q_2$ replaced by $\widetilde{Q}_2$, where
\begin{equation}
2\widetilde{Q}_2 =2Q_2 +\varepsilon |\partial_2H_2^+|_{x_1=0}|^2.
\label{tildQ2}
\end{equation}
Let us now consider the counterpart of equation \eqref{eqH} for \eqref{rb48b}:
\begin{equation}
R = -\partial_0H_N^+ +\underline{\varepsilon\partial_1H_N^+ }\\ +{\rm coeff}\,v_N^+ +{\rm coeff}\,v_2^+ +{\rm coeff}\,H_N^+ +F_N^+\qquad \mbox{on}\ \partial\Omega_T,
\label{reqH}
\end{equation}
where we have the additional underlined term. In view of \eqref{b43b}, \eqref{reqH} implies
\begin{equation}
R = -\partial_0H_N^+ -\underline{\varepsilon\partial_2H_2^+ } +{\rm coeff}\,v_N^+ +{\rm coeff}\,v_2^+ +{\rm coeff}\,H_N^+ +F_N^+ +\underline{\varepsilon F_7^+}\qquad \mbox{on}\ \partial\Omega_T,
\label{reqH'}
\end{equation}
where we have the two additional underlined terms compared to \eqref{eqH}.

By virtue of \eqref{tildQ2} and \eqref{reqH'},
\[
2\widetilde{Q}_2 =2Q_2 + \varepsilon P+ \varepsilon\, {\rm coeff}\,F_7^+|_{x_1=0}\partial_2\varphi + \varepsilon\, {\rm coeff}\,F_7^+|_{x_1=0}\varphi ,
\]
where $Q_2$ is now given in \eqref{Q2"} (it is not the same as in \eqref{tildQ2}) and
\[
P= |\partial_2H_2^+|_{x_1=0}|^2+ {\rm coeff}\,\partial_2H_2^+|_{x_1=0}\partial_2\varphi +  {\rm coeff}\,\partial_2H_2^+|_{x_1=0}\varphi .
\]
By the Young inequality we estimate:
\[
P \geq \frac{1}{2} |\partial_2H_2^+|_{x_1=0}|^2 - C((\partial_2\varphi )^2 + \varphi ^2)\geq - C((\partial_2\varphi )^2 + \varphi ^2).
\]
Using then the last inequality and estimate \eqref{F7tr} we get \eqref{IntQ2} and \eqref{Mt} with $\widetilde{Q}_2$ instead of $Q_2$ in the left-hand side of \eqref{IntQ2}, where the constant $C$ in \eqref{Mt} does not depend on $\varepsilon$. The rest of the arguments are the same as in Section \ref{sec:4} and for problem \eqref{rb48b}--\eqref{rb51b} we obtain the uniform in $\varepsilon$  estimate \eqref{main_est'}.

Returning to the original regularized problem \eqref{r48^}--\eqref{r51^} with nonhomogeneous boundary conditions, we obtain for it the uniform in $\varepsilon$ a priori estimate \eqref{main_est}.

\begin{lemma}
Let the basic state \eqref{a21} satisfies assumptions \eqref{a5}--\eqref{avn} (in the 2D planar case) together with condition \eqref{RTL}. Then for any sufficiently small constant $\varepsilon >0$ and for all $f^{\pm} \in H^1(\Omega_T)$ and $g\in H^{3/2}(\partial\Omega_T)$ which vanish in the past the a priori estimate
\begin{equation}
\sum_{\pm}\|\dot{U}^{\pm \varepsilon}\|_{H^{1}(\Omega_T)}+\|\varphi^{\varepsilon}\|_{H^1(\partial\Omega_T)} \leq C\biggl\{\sum_{\pm}\|f^\pm \|_{H^{1}(\Omega_T)}+ \|g\|_{H^{3/2}(\partial\Omega_T)}\biggr\}
\label{rmain_est}
\end{equation}
holds for problem \eqref{r48^}--\eqref{r51^}, where $C=C(K,\bar{\rho}_0,\bar{p}_0,\bar{\kappa}_0,\epsilon,T)>0$ is a constant independent of $\varepsilon$ and the data $f^\pm$ and $g$.
\label{lem3}
\end{lemma}

We are now in a position to pass to the limit as $\varepsilon \rightarrow 0$. In view of Lemma \ref{lem2}, for all sufficiently small $\varepsilon$  problem \eqref{r48^}--\eqref{r51^} admits a unique solution with the regularity described in Theorem \ref{t1}. Due to the uniform a priori estimate \eqref{rmain_est} we can extract a subsequence weakly convergent to $((\dot{U}^+, \dot{U}^-) ,\varphi )\in H^1(\Omega_T)\times H^1(\partial\Omega_T)$ with $(\dot{U}^+, \dot{U}^-)|_{x_1=0}\in H^{1/2}(\partial\Omega_T)$. Passing to the limit in \eqref{r48^}--\eqref{r51^} as $\varepsilon\to 0$ immediately gives that
$(\dot{U}^+, \dot{U}^- ,\varphi )$ is a solution to problem \eqref{b48^}--\eqref{b51^}. The a priori estimate \eqref{main_est} implies its uniqueness. The proof of Theorem \ref{t1} is complete.

%

%

%

%

\end{document}